\documentclass[11pt]{article}
\usepackage{amssymb,amsmath,amsthm}
\usepackage{amsmath,amssymb,amsthm,amscd,url}
\usepackage{amsmath,amsfonts,amssymb,amscd,amsthm,epsfig}
\usepackage{graphicx}

\newtheorem{prop}{PROPOSITION}[section]
\newtheorem{lemma}{LEMMA}[section]
\newtheorem{cor}{COROLLARY}[section]

\newtheorem{thm}{THEOREM}[section]
\theoremstyle{assumptions}
\newtheorem{assmpn}{Assumption}[section]
\newtheorem{defn}{Definition}[section]
\DeclareMathOperator{\Var}{Var}
\DeclareMathOperator{\Cov}{Cov}
\allowdisplaybreaks[1]

\begin{document}

\numberwithin{equation}{section}

\title{Space-Time Current Process for Independent Random Walks in One Dimension}
 
\author{Rohini Kumar\\
 Department of Mathematics\\
 University of Wisconsin\\
 Madison, WI 53706\\
 {\tt kumar@math.wisc.edu}}
 
\maketitle

 \begin{abstract}
In a system made up of independent random walks, fluctuations of order $n^{1/4}$ from the hydrodynamic limit come from particle current across characteristics. We show that a two-parameter space-time particle current process converges to a two-parameter Gaussian process. These Gaussian processes also appear as the limit for the one-dimensional random average process. The final section of this paper looks at large deviations of the current process.
 
 \end{abstract}

\bigskip
 


\bigskip

\DeclareGraphicsExtensions{.eps,.pdf,.jpg} 
\def\tr{(n^{-\beta}[n^\beta t],n^{-\beta}[n^\beta r])} 
\numberwithin{equation}{section} 
\def\i{$[n^{1-\beta}k_1,n^{1-\beta}(k_1+1)]$} 
 
\def\t{$[n^{1-\beta}k_1,n^{1-\beta}(k_1+1)]$} 
\def\I{$([ny]+[n^{-\beta}k_2\sqrt{n}],[ny]+[r\sqrt{n}]]$} 
\def\J{$([r_0\sqrt{n}]+[nvt_0],[r\sqrt{n}]+[nvt]]$} 
\def\II{$\left[[r_0\sqrt{n}\ ],[(r_0+n^{-\beta})\sqrt{n}\ ]\right]$} 
\def\JJ{\left[[r_0\sqrt{n}\ ]+[nvt_0],[(r_0+n^{-\beta})\sqrt{n}\ ]+[nv(t_0+n^{-\beta})]\right]} 
\def\r{$([ny]+[n^{-\beta}[n^\beta r]\sqrt{n}],[ny]+[r\sqrt{n}]]$} 
\def\rt{$([ny]+[n^{-\beta}[n^\beta r]\sqrt{n}]+[n^{1-\beta}v[n^\beta t],[ny]+[r\sqrt{n}]+[nvt]]$} 
\def\tt{$[n^{1-\beta}[n^\beta t],nt]$} 
\bigskip

\section{Introduction} It is well known that particle systems that appear
different at the microscopic level often behave almost identically at a macroscopic
level. This has been observed in the hydrodynamic limits and fluctuation
results of several particle models. Consequently, there is much to be gained
in studying the behavior of simpler stochastic particle systems in the hope
that at the macroscopic level they will reflect the behavior of a universal
class of systems. While the hydrodynamic limit of several models have been
studied, fluctuation results have proved elusive for many systems. In this
paper we consider particle current fluctuations in the one dimensional
independent random walk model.

The hydrodynamic limit for particle distribution in 
typical asymmetric systems are solutions to p.d.e's of the
form 
\begin{equation}
\label{pde}
 \partial_t u+\partial_x f (u)=0. \end{equation}
 In the case of nearest neighbor Totally Asymmetric Simple Exclusion Process (TASEP) in one
dimension, the flux function $f (\rho) = \rho (1 - \rho)$. For non-interacting particle systems
$f (\rho) = v \cdot \rho$ where $v$ is the average velocity of the particles. Thus the relevant p.d.e. for a system of independent asymmetric random walks is 
 \begin{equation}\label{transport_eqn}\partial_tu+v\cdot \partial_xu=0\end{equation} (Prop 3.1, page 15 in \cite{KL}).

 From the transport equation (\ref{transport_eqn}) we see that in the independent random walk model, the initial density profile shifts with velocity $v$. Consider an observer moving at constant velocity $v$. The path of the observer is a characteristic line of the transport equation (\ref{transport_eqn}). It is natural to expect the net current of particles across the path of the observer to be zero. But what are the fluctuations in this particle current? This is the question we address in this paper. It has been observed that these current fluctuations are of order $n^{1/4}$ \cite{SEPP}.  Here, $n$ is the scaling parameter. Typically we scale both space and time  by $n$ in asymmetric models, this is called Euler scaling. In  symmetric models we use diffusive scaling i.e. we scale time by $n$ and space by $\sqrt{n}$. There is a general belief that when $f$ is linear (i.e. $f''\equiv  0$) in (\ref{pde}), the fluctuations in particle current across characteristics of (\ref{pde}) should be of order $n^{1 / 4}$. This has been shown for
the random average process (RAP) and for the one dimensional 
independent random
walk model where $f'' \equiv 0$. 

In this paper we study both the fluctuations and the large 
deviations of the   current process for independent
walks.  For the fluctuations we consider the 
current process  indexed both by 
time and spatial shifts of order
 $\sqrt{n}$  of characteristic lines.  The  $\sqrt{n}$ order
for spatial scaling is the natural one because the individual
random walks fluctuate on that scale.  
We extend the distributional limit of 
 \cite{SEPP} to a process limit for the space-time current process. 
The space-time current process was also studied for RAP 
in \cite{BRS} but only convergence of  finite dimensional 
distributions was shown without process-level tightness. 
The same family of Gaussian processes arises as limits 
for both  RAP and  independent random walks. 

It is interesting to note that there are models which are not asymmetric yet
exhibit subdiffusive current fluctuations with Gaussian scaling limits. It was
conjectured (conjecture 6.5 in \cite{Spohn}) that
subdiffusive fluctuations in 1 dimensional nearest neighbor symmetric simple  exclusion processes (SSEP) converge to
fBM with Hurst parameter $1 / 4$. It was subsequently proved in the finite
dimensional distributions sense and has recently been proved in the full
functional central limit theorem sense in \cite{Sethuraman}. This
says that the universality class of
current fluctuations of order $n^{1 / 4}$  contains both symmetric and asymmetric processes. However, the symmetric and asymmetric processes differ on the level of hydrodynamics. 

This paper is organized as follows. We start with a description of
the independent random walk model and the statements of the main results in section \ref{sec:model}.
The next three sections \ref{sec:finitedim}, \ref{sec:tightness} and  \ref{sec:largedev} cover the proofs. Section \ref{sec:finitedim} gives the convergence of finite dimensional distributions and section \ref{sec:tightness} proves process level tightness. A note on the tightness methods used here: since we are
interested in a two-parameter process, the standard theorems on convergence in
$D_{\mathbb R}[0,\infty)$ and $C_{\mathbb R}[0,\infty)$ spaces do not apply. We appeal to two papers,
\cite{BW} and \cite{Durr}, that provide suitable criteria for deducing tightness.
\cite{BW} gives the context in which we speak of convergence for
two-parameter processes and a tightness criterion. The proof of Proposition 5.7 in \cite{Durr} is extended to two dimensions to prove the tightness criterion. The last section 
contains proofs of some large deviation results for the current process.

\section{Model and results}
\label{sec:model}
\subsection{Independent random walk model}
Consider particles distributed over the one dimensional integer lattice which evolve like independent continuous-time random walks. We are given the initial occupation variables $\eta_0=\{\eta_0(x):x\in\mathbb Z\}$ defined on some probability space. Let $X_{m,j}(t)$ denote the position at time $t$ of the $j$th random walk 
starting at site $m$. 
The common jump rates of the random walks are given by a probability kernel $\{p(x):x\in\mathbb Z\}$. Once the initial positions of the random walks are specified, their subsequent evolutions $\{X_{m,j}(t)-X_{m,j}(0):m\in\mathbb Z, j=1,\hdots, \eta_0(m)\}$ are as i.i.d. random walks starting at the origin, on the same probability space, independent of $\eta_0$. Define \[\eta_t(x):=\sum_{m\in\mathbb Z}\sum_{j=1}^{\eta_0(m)}{\bf 1}\{X_{m,j}(t)=x\}\] to be number of particles on site $x$ at time $t$. 

\begin{assmpn}
\label{assmpn1}
For the random walk kernel, we assume that, for some $\delta>0$,
\begin{equation}  \sum_{x\in\mathbb Z}e^{\theta x}p(x)<\infty \text{ for }|\theta|\leq\delta 
\end{equation} 
(This assumption will enable us to calculate large deviation bounds for the random walks.)
\end{assmpn}
Throughout this paper, we assume that $\mathbb N$ denotes the set of positive integers.
Let $\{\eta_0^n: n\in \mathbb N\}$ be a sequence of initial occupation variables defined on some probability space.
\begin{assmpn}
\label{assmpn3}
 For each n, the initial occupation variables $\{\eta_0^n(x):x\in\mathbb Z\}$ are independent. They have  a uniformly bounded twelfth moment:\begin{equation}
\label{momentassmpn}\sup_{n\in N,x\in\mathbb Z}E[\eta_0^n(x)^{12}]<\infty.\end{equation}
Let $\rho_0^n(x)=E\eta_0^n(x) \textrm{ and } v_0^n(x)=Var[\eta_0^n(x)]$
 be the mean and variance resp. of the initial occupation variable $\eta_0^n(x), x\in\mathbb Z$. Let $\rho_0$ and $v_0$ be two given nonnegative, finite numbers.  The means $\rho_0^n$ and variances $v_0^n$ approximate $\rho_0$ and $v_0$ in the following sense:
   There exist positive integers $L=L(n)$ such that $n^{-1/4}L(n)\to 0$ and for any finite constant A,
\begin{equation}\label{initial_cond}
\lim_{n\to\infty} \sup_{|m|\leq A\sqrt{n\log n}}n^{1/4}\Bigl|\frac{1}{L(n)}\sum_{j=1}^{L(n)}\rho_0^n(m+j)-\rho_0\Bigr|=0\end{equation} \\
 The same assumption holds when $\rho_0^n$ and $\rho_0$ are replaced by $v_0^n$ and $v_0$. 
 \end{assmpn}
 As in \cite{SEPP}, the reason for the complicated assumption (\ref{initial_cond}) is to accommodate both random and deterministic initial conditions. For random $\eta_0^n(x)$ we could take $\rho_0^n(x)=\rho_0(\frac{x}{n})$ for some sufficiently regular function $\rho_0(\cdot)$. However, for deterministic $\eta_0^n(x)$ we cannot do this unless $\rho_0(x)$ is integer-valued. A couple of examples illustrating random and deterministic initial configurations that satisfy assumptions \ref{assmpn1} and \ref{assmpn3} can be found in \cite{SEPP}. 

Let $$v=\sum_xxp(x)\textrm{ and }\kappa_2=\sum_xx^2p(x).$$
 
The characteristics of (\ref{transport_eqn}) are straight lines with slope $v$. Fix $T>0$ and $S>0$. For $t\in[0,T]$ and $r\in[-S,S]$,
we let $Y_n(t,r)$ denote the net right-to-left particle 
current  during time $[0,nt]$ across the characteristic line
 starting at $([r\sqrt{n}],0)$.  
More precisely, 
\begin{equation}
\label{Y2}
\begin{split}
Y_n(t,r)&:= 
\sum_{m=-\infty}^\infty\sum_{j=1}^{\eta_0^n(m)}
\bigl[\mathbf{1} \{X_{m,j}(nt)\leq [nvt]+[r\sqrt{n}]\}
\mathbf{1} \{m>[r\sqrt{n}]\}\\
&\qquad\qquad -\mathbf{1} \{X_{m,j}(nt)>[nvt]+[r\sqrt{n}]\}
\mathbf{1} \{m\leq [r\sqrt{n}]\}\bigr]
\end{split}
\end{equation}
where $X_{k,j}(\cdot)$ is the $j$th random walk that starts at site $k$. Note that the random walks denoted as $X_{k,j}$ in the definition of $Y_n(t,r)$ should actually be $X^n_{k,j}$, but we drop the superscript $n$ for notational simplicity. 

\subsection{Distributional limit}
We give a brief description of the path space of the process $Y_n(\cdot,\cdot)$. %We denote this space by $D_2$. 
Let $D_2=D_2([0,T]\times [-S,S],\mathbb R)$ be the space of 2-parameter cadlag functions with Skorohod's topology. Let $Q:=[0,T]\times[-S,S]$. For any $(t,r)\in Q$, we can divide $Q$ into four quadrants: \[Q^1_{(t,r)}:=\{(s,q)\in Q:s\geq t, q\geq r\},\ Q^2_{(t,r)}:=\{(s,q)\in Q:s\geq t, q< r\},\] \[ Q^3_{(t,r)}:=\{(s,q)\in Q:s< t, q< r\},\ Q^4_{(t,r)}:=\{(s,q)\in Q:s< t, q\geq r\}.\]
Then the precise definition of $D_2$ is 
\begin{align*}D_2=\{&f:Q \to \mathbb R\text{ such that for any point }(t,r)\in Q, 
 \lim_{\substack{(s,q)\in Q^i_{(t,r)}\\(s,q)\to (t,r)}}f(s,q)\\&\text{ exists for }i=1,2,3,4\text{ and }\lim_{\substack{(s,q)\in Q^1_{(t,r)}\\(s,q)\to (t,r)}}f(s,q)=f(t,r) \}.\end{align*}
 In other words,
 $D_2$ contains functions that are continuous from the right and above with limits from the left and below. Skorohod's topology in $D_{\mathbb R}[0,\infty)$ is extended to this space in the most natural way.
The space of multiparameter cadlag functions and their topology is described in detail in \cite{BW}. 
By Theorem 2 in \cite{BW}, a sufficient criterion for the weak convergence  $X_n \to X$ in $D_2$ is,
\begin{enumerate}
\item
For all finite subsets $\{(t_i,r_i)\}\subset [0,T]\times [-S,S]$,   
\[(X_n(t_1,r_1),\cdots,X_n(t_N,r_N)) \to 
(X(t_1,r_1),\cdots,X(t_N,r_N)) \]
weakly,  and
\item
$\lim_{\delta\to 0}\limsup_n P\{w_{X_n}(\delta)\geq \epsilon\}=0$ for all $\epsilon>0$,
where the modulus of continuity is defined by
\[w_x(\delta)=\sup_{\substack{(s,q),(t,r)\in[0,T]\times [-S,S]\\|(s,q)-(t,r)|<\delta}}|x(s,q)-x(t,r)|.\]
\end{enumerate}
Clearly, $\{Y_n(t,r):t\in[0,T], r\in [-S,S]\}$ are $D_2$-valued processes and we can use the above criterion  to prove their convergence in the weak sense.

Denote the centered Gaussian density and distribution
 with variance $\sigma^2$
 by 
$$\phi_{\sigma^2}(x)=\frac{1}{\sqrt{2\pi\sigma^2}}\exp\{-\frac{1}{2\sigma^2}x^2\}
\ \text{ and }\ 
\Phi_{\sigma^2}(x)=\int_{-\infty}^x\phi_{\sigma^2}(y)dy.
$$
Define also
 $$\Psi_{\sigma^2}(x)=\sigma^2\phi_{\sigma^2}(x)-
x(1-\Phi_{\sigma^2}(x)).$$  
For $(s,q),(t,r) \in [0,T]\times [-S,S]$, define two covariances by
\begin{equation}
\label{gamma0}
\Gamma_0((s,q),(t,r))=\Psi_{\kappa_2s}(|q-r|)+\Psi_{\kappa_2t}(|q-r|)-\Psi_{\kappa_2(t+s)}(|q-r|)
\end{equation}
and
\begin{equation}
\label{gammaq}
\Gamma_q((s,q),(t,r))=\Psi_{\kappa_2(t+s)}(|q-r|)-\Psi_{\kappa_2|t-s|}(|q-r|).
\end{equation}

\begin{thm}
\label{main_thm}
Define $Y_n(t,r)$ as in (\ref{Y2}). Then under Assumptions \ref{assmpn1} and \ref{assmpn3}, as $n\to\infty$, the process $n^{-1/4}Y_n(\cdot,\cdot)$ converges weakly on the space $D_2$ to the mean zero Gaussian process $Z(\cdot,\cdot)$ with covariance 
\begin{equation}
\label{cov}
EZ(s,q)Z(t,r)=\rho_0\Gamma_q((s,q),(t,r))+v_0\Gamma_0((s,q),(t,r)).
\end{equation}
\end{thm}
{\bf Note:} We will show later that $n^{-1/4}EY_n(t,r)\to 0$ as $n\to\infty$ uniformly for $t\in[0,T]$ and $r\in[-S,S]$. Hence, in the above theorem we can replace $n^{-1/4}Y_n$ with the centered current process with impunity.
\begin{cor}
Under the invariant distribution where $\{\eta_t^n(x):x\in\mathbb Z\}$ are i.i.d.\ Poisson with mean $\rho$ for all $n$, $n^{-1/4}Y_n(\cdot,\cdot)$ converges weakly in $D_2$ to a mean-zero Gaussian process $Z(\cdot,\cdot)$ with covariance
\[ EZ(s,q)Z(t,r)=\rho\bigl(\Psi_{\kappa_2s}(|q-r|)+\Psi_{\kappa_2t}(|q-r|)-\Psi_{\kappa_2|t-s|}(|q-r|)\bigr).\]
In particular, for a fixed $r$ the process $\{Z(t,r):t\in [0,T]\}$ has covariance
\[EZ(s,r)Z(t,r)=\rho\sqrt{\frac{\kappa_2}{2\pi}}\bigl(\sqrt{s}+\sqrt{t}-\sqrt{|t-s|}\bigr),\]
i.e., process $Z(\cdot,r)$ is fractional Brownian motion with Hurst parameter $1/4$.
\end{cor}
The covariance (\ref{cov}) is the same as the covariance of the limiting Gaussian process in the RAP model found in \cite{BRS}, with different coefficients in front of $\Gamma_q$ and $\Gamma_0$. As in \cite{BRS}, we can represent the Gaussian process $Z(\cdot,\cdot)$ as a stochastic integral:  
\begin{equation}
\label{sde}
\begin{split}
Z(t,r)=\sqrt{\kappa_2\rho_0}&\int_{[0,t]\times\mathbb R}\phi_{\kappa_2(t-s)}(r-z)dW(s,z)\\&
+\sqrt{v_0}\int_{\mathbb R}sgn(x-r)\Phi_{\kappa_2t}(-|x-r|)dB(x).
\end{split}
\end{equation}
The equality in (\ref{sde}) is equality in distribution. $W$ is a two-parameter Brownian motion defined on $\mathbb R_+\times\mathbb R$, $B$ is a one-parameter Brownian motion defined on $\mathbb R$, and $W$ and $B$ are independent of each other. The stochastic integral clearly delineates the two sources of fluctuations in the current. The first integral represents the space-time noise created by the dynamics, and the second integral represents the initial noise propagated by the evolution.

\subsection{Large deviation results}
\def\P{\Phi_{\kappa_2t}}
\def\Z{\P(y)+e^{\alpha}(1-\P(y)) }
\def\z{e^{-\alpha}\P(y)+1-\P(y)}
We first state large deviation results for $Y_n(t,r)$ with fixed $r\in\mathbb R$ and $t>0$.
\begin{assmpn}\label{assmpnld}
Assume the initial occupation variables 
$\{\eta_0^n(m):m\in\mathbb{Z}\}$ are i.i.d. 
Let 
\[\gamma(\theta)=\log Ee^{\theta\eta_0^n(m)}\]
with effective domain 
\[ D_\gamma:=\{\alpha\in \mathbb R:\gamma(\alpha)<\infty\}.\] 
Assume
$D_\gamma=\mathbb R$.
\end{assmpn}
For $\lambda\in\mathbb R$, define
\[
Z_\lambda(y):=
\begin{cases}
\log\{\Phi_{\kappa_2t}(y)+e^\lambda(1-\Phi_{\kappa_2t}(y))\}&\text{ for } y> 0\\
\log\{e^{-\lambda}\Phi_{\kappa_2t}(y)+1-\Phi_{\kappa_2t}(y)\}&\text{ for }y\leq 0 
\end{cases}
\]

and
 \[\Lambda(\lambda):=\int_{-\infty}^\infty\gamma(Z_\lambda(y))dy\]
with effective domain 
\[
D_\Lambda:=\{\alpha\in \mathbb R:\Lambda(\alpha)<\infty\}.
\]
$\Lambda$  turns out to be the limiting moment generating 
function of the current.

Recall
\begin{defn}
A convex function $\Lambda :\mathbb R\to(-\infty,\infty]$ is essentially smooth if:\\
a) $D_\Lambda^o$(interior of $D_\Lambda$) is non-empty.\\
b) $\Lambda(\cdot)$ is differentiable throughout $D_\Lambda^o$.\\
c) $\Lambda(\cdot)$ is steep, i.e., $\lim_{n\to\infty}|\nabla\Lambda(\lambda_n)|=\infty$ whenever $\{\lambda_n\}$ is a sequence in $D_\Lambda^o$ converging to a boundary point of $D_\Lambda^o$.
\end{defn}

Let $\gamma^*(x):=\sup_{\lambda\in \mathbb R}\{\lambda\cdot x-\gamma(\lambda)\}$ be the convex dual of $\gamma(\cdot)$.
Let \[I(x):=\sup_{\lambda\in \mathbb R}\{\lambda\cdot x-\Lambda(\lambda)\}.\]
Recall the usual definition of Large Deviation Principle (LDP). The sequence of random variables $\{X_n\}$ satisfies the LDP with rate function $J(x)$ and normalization $\{\sqrt{n}\}$ if the following are satisfied:
\begin{enumerate}
  \item For any closed set $F\subset \mathbb R$,\[\limsup_{n\to\infty}\frac{1}{\sqrt{n}}\log P\{X_n\in F\}\leq -\inf_{x\in F}J(x)\]
  \item For any open set $G\subset \mathbb R$, \[\liminf_{n\to\infty}\frac{1}{\sqrt{n}}\log P\{X_n\in G\}\geq -\inf_{x\in G}J(x)\]
 \end{enumerate}
The rate function $J$ is said to be a good rate function if its level sets are compact.

\def\a{{\alpha(x)}}
Let $Br_p(\lambda):=\log E e^{\lambda X}$, where $X\sim$Bernoulli($p$), be the logarithmic moment generating function of Bernoulli random variables. Its convex dual is 
\[Br_p^*(x)=x\log\frac{x}{p}+(1-x)\log\frac{1-x}{1-p}\] 
for $0\le x\le 1$.
Define $\alpha(x)$ implicitly by  
\begin{equation}
\label{alpha(x)}
x=\Lambda'(\a).
\end{equation}
This definition is well defined for all $x\in\mathbb R$ since $\Lambda(\cdot)$ is strictly convex.
For any $\alpha\in\mathbb R$, define \begin{equation}\label{F_x}F_\alpha(y):=\frac{e^{-\alpha}\P(y)}{1-\P(y)+e^{-\alpha}\P(y)} \text{ for }y\in\mathbb R.\end{equation}
By \eqref{alpha(x)} we have
\begin{equation}
\label{current_x}
x=\int_0^\infty\gamma' (Z_{\alpha(x)}(y) )(1-F_{\alpha(x)}(y))dy-\int_{-\infty}^0\gamma' (Z_{\alpha(x)}(y) )F_{\alpha(x)}(y)dy.
\end{equation}

Define \begin{equation}\label{I_1}I_1(x):= \int_{-\infty}^\infty\gamma^*\left\{\gamma'(Z_\a(y))\right\}dy
\end{equation}

and
\begin{equation}
\label{I_2}I_2(x):=\int_{-\infty}^\infty \gamma'(Z_\a(y))Br^*_{\P(y)}(F_\a(y))dy. \end{equation}
Recall that we assumed $\eta_0^n(\cdot)$ are i.i.d. at the beginning of this section. Consequently, the  underlying distribution is shift invariant. The rate function in Theorem 2.2 therefore does not depend on $r$, as  the marginals of the current process, $Y_n(t,r)$, are shift invariant. 
\begin{thm}
\label{ld}
Let Assumptions \ref{assmpn1} and \ref{assmpnld} hold.
For fixed real $r$ and $t>0$, $n^{-1/2}Y_n(t,r)$ satisfies the large deviation principle with normalization $\{\sqrt{n}\}$ and  
 good, strictly convex  rate function
\begin{equation}
\begin{split}
\label{RateFn}
I(x)=I_1(x)+I_2(x),\quad x\in\mathbb R.  
\end{split}
\end{equation}
\end{thm}
A few words on the rate function. $I$ has a unique zero at zero. The rate function $I$ balances two costs: the cost of deviations in the initial occupation variables, given by $I_1$, and the cost of deviations in the probability with which particles cross the characteristic lines, given by $I_2$.  In the macroscopic picture, $1-\P(y)$ is the a priori probability with which particles initially at distance $y>0$ from the characteristic line cross it by time $t$, while a particle at distance $y\leq 0$ crosses the characteristic line with probability $\P(y)$.  

An intuitive understanding of the LDP is as follows. To allow a current of size $x$ at time $t$  the system deviates in such a way that its behavior is governed by a new probability measure. Under this new probability measure  the mean number of particles initially at site $y$ is $\gamma'(Z_{\alpha(x)}(y))$ and the probabilities  $1-\P(y)$ and $\P(y)$ are tilted to give new probabilities  $1-F_{\alpha(x)}(y)$ and $F_{\alpha(x)}(y)$.
 The term $Br_{\P(y)}(F_{\alpha(x)}(y))$ measures the cost of the deviation of the probability $1-\P(y)$ (or $\P(y)$) to $1-F_{\alpha(x)}(y)$ (or $F_{\alpha(x)}(y)$).  The new measure depends on $\alpha(x)$ which is chosen so that the mean current under the new measure is $x$, this is evident from \eqref{current_x}.

 The modus operandi for the proof involves using non-rigorous, intuitive ideas for finding a candidate for the rate function and then checking if it is, in fact, the correct rate function. 

We give explicit formulas for the two simplest cases:
the stationary situation with i.i.d.\ Poisson occupations, 
and the case of  deterministic initial occupations.

\begin{cor}
\begin{enumerate}
\item If $\eta_0^n(\cdot)\sim Poisson(\rho)$, then under assumptions \ref{assmpn1} and \ref{assmpnld}, the rate function is 
\begin{equation}
\label{rate_fn}
\begin{split}
I(x)=x\log\biggl(\frac{x\sqrt{\pi}}{\rho\sqrt{2\kappa_2t}}+\sqrt{1+\frac{x^2\pi}{2\rho^2\kappa_2t}}\biggr)-\rho\sqrt{\frac{2\kappa_2t}{\pi}}\biggl(\sqrt{1+\frac{x^2\pi}{2\rho^2\kappa_2t}}-1\biggr)\text{ for }x\in\mathbb R.
\end{split}
\end{equation}
\item If $\eta_0^n(\cdot)\equiv 1$, then under assumption \ref{assmpn1} the rate function is
\[I(x)=\int_{-\infty}^\infty\biggl[(1-F_\a(y))\log\left(\tfrac{1-F_\a(y)}{1-\P(y)}\right)+F_\a(y)\log\left(\tfrac{F_\a(y)}{\P(y)}\right)\biggr]dy\]
where $F_\a(y)=\frac{e^{-\a}\P(y)}{1-\P(y)+e^{-\a}\P(y)}$ with $\a$ chosen so that 
\[x=\int_0^\infty(1-F_\a(y))dy-\int_{-\infty}^0F_\a(y)dy.\]
\end{enumerate}
\label{ldpcor}
\end{cor}

For the process level, under the stationary distribution, we can show an abstract LDP by applying a theorem from \cite{KF}. But currently we do not have an attractive representation of the rate function.  The rate function is given in terms of a variational expression in \eqref{rate_process}.
\begin{thm}
\label{LDP_process}
If $\eta_0^n(\cdot )\sim Poisson(\rho)$, then under assumptions \ref{assmpn1} and \ref{assmpnld}  the sequence of processes $\{n^{-1/2}Y_n(\cdot,0)\}$ in $D_{\mathbb R}[0,\infty)$ satisfies a LDP with a good rate function.
\end{thm}

\bigskip

There exist some large deviation results for independent random walk systems in the literature \cite{Lee},\cite{CD},\cite{CG}. These 
papers essentially deal with large deviations of occupation times of sites.  Lee \cite{Lee} and Cox and Durrett \cite{CD}
 find LDP's for weighted occupation times of sites under 
deterministic and stationary (i.i.d.\ Poisson)
 initial distributions.  Even the normalizations of the LDP's
were distinct for these two cases.  This is in contrast to 
our current large deviations in Theorem \ref{ld}  where
the normalizations for the Poisson and the deterministic
case were the same. 

\section{\bf Weak Convergence of Finite Dimensional Distributions}\label{sec:finitedim}
We begin the proof of Theorem \ref{main_thm} by first showing weak convergence of the finite dimensional distributions of $n^{-1/4}Y_n(\cdot,\cdot)$.
\begin{prop}
\label{finite_dim}
Define $Y_n(t,r)$ as above. Then under assumptions \ref{assmpn1} and \ref{assmpn3} the finite dimensional distributions of the processes $\{Y_n(t,r):(t,r)\in[0,T]\times[-S,S]\}$ converge weakly as $n\to\infty$ to the finite-dimensional distributions of the mean zero Gaussian process $\{Z(t,r):(t,r)\in[0,T]\times[-S,S]\}$ with covariance given in (\ref{cov}).
\end{prop} 
To prove convergence of finite dimensional distributions we use Lindeberg-Feller and check the conditions of Lindeberg-Feller by brute force. We show that the expected value of the process converges uniformly to $0$ and hence we can consider the centered process when proving convergence. We also appeal to large deviations of random walks to control the contributions to the current process from distant particles.

Fix $N$ space-time points: 
$(t_1,r_1),(t_2,r_2),\ldots,(t_N,r_N)$ where $(t_i,r_i)\neq(t_j,r_j)$ for $i\neq j$.
We will prove that as $n\to\infty$ the vector \[n^{-1/4}(Y_n(t_1,r_1),Y_n(t_2,r_2),...,Y_n(t_N,r_N))\]converges in distribution to the mean-zero Gaussian random vector \[(Z(t_1,r_1),Z(t_2,r_2),...,Z(t_N,r_N))\] with covariance 
$$EZ(t_i,r_i)Z(t_j,r_j)=\rho_0\Gamma_q((t_i,r_i),(t_j,r_j))+v_0\Gamma_0((t_i,r_i),(t_j,r_j)).$$

Let $\theta=(\theta_1,...,\theta_N)\in \mathbf R^N$ be  arbitrary.
Recall that $X_{k,j}(\cdot)$ is the $j$th random walk that starts at site $k$. 
\begin{align*}&n^{-1/4}\sum_{i=1}^N\theta_iY_n(t_i,r_i)\\&=n^{-1/4}\sum_{i=1}^N\theta_i\biggl\{\sum_{m=-\infty}^\infty\sum_{j=1}^{\eta_0^n(m)}\bigl(\mathbf{1} \{X_{m,j}(nt_i)-X_{m,j}(0)\leq [nvt_i]+[r_i\sqrt{n}]-m\}\mathbf{1} \{m>[r_i\sqrt{n}]\}\\&\qquad-\mathbf{1} \{X_{m,j}(nt_i)-X_{m,j}(0)>[nvt_i]+[r_i\sqrt{n}]-m\}\mathbf{1} \{m\leq [r_i\sqrt{n}]\}\bigr)\biggr\}\end{align*}

Let \[A^{t,r}_{m,j}=\{X_{m,j}(nt)-X_{m,j}(0)\leq [nvt]+[r\sqrt{n}]-m\}\] and  \[A^{t,r}=\{X(nt)\leq[nvt]+[r\sqrt{n}]-m\}\] where $X(\cdot)$ represents a random walk with rates $p(x)$ starting at the origin.
The evolution of the random walks is independent of their initial occupation numbers $\eta^n_0(x)$.
Define \begin{align*}U_m(t,r):=&n^{-1/4}\sum_{j=1}^{\eta_0^n(m)}\bigl\{\mathbf{1} \{A^{t,r}_{m,j}\}\mathbf{1} \{m>[r\sqrt{n}]\}-\mathbf{1} \{(A^{t,r}_{m,j})^c\}\mathbf{1} \{m\leq [r\sqrt{n}]\}\bigr\}\\
&-n^{-1/4}\rho_0^n(m)\bigl\{P\{A^{t,r}\}\mathbf{1} \{m>[r\sqrt{n}]\}-P\{(A^{t,r})^c\}\mathbf{1} \{m\leq [r\sqrt{n}]\}\bigr\}\end{align*} and 
\[\bar{U}_m:=\sum_{i=1}^N\theta_iU_m(t_i,r_i).\]
Then we can write
\begin{align*}n^{-1/4}\sum_{i=1}^N\theta_iY_n(t_i,r_i)=&n^{-1/4}\sum_{i=1}^N\theta_i\{Y_n(t_i,r_i)-EY_n(t_i,r_i)\}+n^{-1/4}\sum_{i=1}^N\theta_i EY_n(t_i,r_i)\\
&=\sum_{m=-\infty}^\infty \bar{U}_m+n^{-1/4}\sum_{i=1}^N\theta_i EY_n(t_i,r_i).\end{align*}

We split $\sum_{m=-\infty}^\infty \bar{U}_m$ into two sums $S_1$ and $S_2$ as follows.
Choose $r(n)$ so that $r(n)=o(\sqrt{\log\ n})$ and $r(n)\to\infty$ slowly enough that \[r(n)H(n^{1/8})\to 0\text{ as }n\to\infty,\]
where $H(M)=\sup_{n\geq 1,x\in \mathbb Z}E[\eta_0^n(x)^21\{\eta_0^n(x)\geq M\}].$
Write
\begin{align*}n^{-1/4}\sum_{i=1}^N\theta_i\{Y_n(t_i,r_i)-EY_n(t_i,r_i)\}&=\sum_{|m|\leq r(n)\sqrt{n}}\bar{U}_m+\sum_{|m|>r(n)\sqrt{n}}\bar{U}_m\\&=:S_1+S_2.\end{align*}
We show that $S_2$ goes to $0$ in $L^2$ and use Lindeberg-Feller for $S_1$ to show that it converges to a mean-zero normal distribution.
\begin{lemma}
\label{s2}
$ES_2^2\to0$ as $n\to\infty$.
\end{lemma}
\begin{proof}
$EU_m=0$, so $E\bar{U}_m=0$.
$S_2$ is therefore a sum of independent mean zero terms $\bar{U}_m$.
By Schwarz Inequality,
\begin{align*}
ES_2^2=\sum_{|m|\geq [r(n)\sqrt{n}]+1} E\bar{U}_m^2\leq\|\theta\|^2\sum_{i=1}^N\sum_{|m|\geq [r(n)\sqrt{n}]+1} EU_m^2(t_i,r_i).\end{align*}
Since $N$ is fixed, it suffices to show for fixed $(t,r)$
\[\lim_{n\to\infty}\sum_{m\geq [r(n)\sqrt{n}]+1}^\infty EU_m^2(t,r)=0.\]
Recall that $A^{t,r}_{m,j}=\{X_{m,j}(nt)-X_{m,j}(0)\leq [nvt]+[r\sqrt{n}]-m\}$ and  $A^{t,r}=\{X(nt)\leq[nvt]+[r\sqrt{n}]-m\}$.
If $T_K=\sum_{i=1}^KZ_i$ is a random sum of i.i.d. summands $Z_i$ independent of random K, then
\[\Var[T_K]=EK\cdot \Var Z_1+(EZ_1)^2\cdot \Var K.\]
For fixed $m$, take $K=\eta_0^n(m)$ and 
\[Z_j=\mathbf{1} \{A^{t,r}_{m,j}\}\mathbf{1} \{m>[r\sqrt{n}]\}-\mathbf{1} \{(A^{t,r}_{m,j})^c\}\mathbf{1} \{m\leq [r\sqrt{n}]\}.\]
Then,
\begin{align*}E[U_m^2(t,r)]=\Var\bigl[n^{-1/4}\sum_{j=1}^{\eta_0^n(m)}\bigl\{\mathbf{1} \{A^{t,r}_{m,j}\}\mathbf{1} \{m>[r\sqrt{n}]\}-\mathbf{1} \{(A^{t,r}_{m,j})^c\}\mathbf{1} \{m\leq [r\sqrt{n}]\}\bigr\}\bigr].\end{align*}
\begin{align*}\Var Z_j=&P(A^{t,r}_{m,j})P((A^{t,r}_{m,j})^c)\mathbf{1}\{ m>[r\sqrt{n}]\}+P(A^{t,r}_{m,j})P((A^{t,r}_{m,j})^c)\mathbf{1}\{ m\leq [r\sqrt{n}]\}\\&
=P(X(nt)\leq[nvt]+[r\sqrt{n}]-m)P(X(nt)>[nvt]+[r\sqrt{n}]-m) .
\end{align*}
\begin{align*}[EZ_1]^2&=P(X(nt)\leq [nvt]+[r\sqrt{n}]-m)^2\mathbf{1}\{m>[r\sqrt{n}]\}\\
&\qquad+P(X(nt)>[nvt]+[r\sqrt{n}]-m)^2\mathbf{1}\{m\leq [r\sqrt{n}]\}.\end{align*}
\begin{align*}E[U_m^2(t,r)]&=n^{-1/2}\rho_0^n(m)P(A^{t,r})P((A^{t,r})^c)+n^{-1/2}v_0^n(m)\bigl[P(A^{t,r})^2\mathbf{1}\{m>[r\sqrt{n}]\}\\&\qquad \qquad \qquad \qquad \qquad \qquad \qquad  +P((A^{t,r})^c)^2\mathbf{1}\{m\leq [r\sqrt{n}]\}\bigr].\end{align*}
Using the uniform bound (\ref{momentassmpn}) on moments
we get 
\begin{align*}E[U_m^2(t,r)]&
\leq n^{-1/2}C\bigl[P(X(nt)\leq[nvt]+[r\sqrt{n}]-m)\mathbf{1}\{m>[r\sqrt{n}]\}\\&
\qquad+P(X(nt)>[nvt]+[r\sqrt{n}]-m)\mathbf{1}\{m\leq [r\sqrt{n}]\}\bigr].
\end{align*}

By standard large deviation theory, for arbitrarily small $\delta$, there is a constant $0<K<\infty$ such that when $m>[r\sqrt{n}]$,
\begin{equation}\label{ld_bounds_rw1}
P(X(nt)\leq [nvt]+[r\sqrt{n}]-m)\leq 
\begin{cases}
 e^{\{-K(m-[r\sqrt{n}])^2/nt\}} & \text{ if } |m-[r\sqrt{n}]|\leq nt\delta\\
 e^{\{-K|m-[r\sqrt{n}]|\}} &\text{ if }|m-[r\sqrt{n}]|>nt\delta
\end{cases}
\end{equation}
and when $m\leq [r\sqrt{n}]$,
\begin{equation}\label{ld_bounds_rw2}
P(X(nt)> [nvt]+[r\sqrt{n}]-m)\leq 
\begin{cases}
 e^{\{-K([r\sqrt{n}]-m)^2/nt\}} & \text{ if } |[r\sqrt{n}]-m|\leq nt\delta\\
 e^{\{-K|[r\sqrt{n}]-m|\}} &\text{ if }|[r\sqrt{n}]-m|>nt\delta.
\end{cases}
\end{equation}
Consequently,
\begin{align*}&\sum_{|m|=[r(n)\sqrt{n}]+1}^\infty E[U_m^2(t,r)]\\&\leq Cn^{-1/2}\sum_{m_1=[r(n)\sqrt{n}]-[r\sqrt{n}]+1}^{[nt\delta]}e^{-Km_1^2/nt}+Cn^{-1/2}\sum_{m_1=[nt\delta]+1}^\infty e^{-Km_1}\\
&\quad+Cn^{-1/2}\sum_{m_1=[r(n)\sqrt{n}]+[r\sqrt{n}]+1}^{[nt\delta]}e^{-Km_1^2/nt}+Cn^{-1/2}\sum_{m_1=[nt\delta]+1}^\infty e^{-Km_1}\\
&\leq C\int_{r(n)-r}^\infty e^{-Kx^2/t}dx+C\int_{r(n)+r}^\infty e^{-Kx^2/t}dx+2\frac{C}{\sqrt{n}}e^{-Knt\delta}.\frac{1}{1-e^{-K}}
\to 0 \end{align*} as $n\to\infty$ (since $r(n)\to\infty$).
Therefore, $ES_2^2\to 0$ as $n\to\infty$.
\end{proof}

Define
\begin{align*}&\Gamma_0^{(2)}((s,q),(t,r))=\int_{q\vee r}^\infty P(\sqrt{\kappa_2}B(s)\leq q-x)P(\sqrt{\kappa_2}B(t)\leq r-x)dx\\&+\int_{-\infty}^{q\wedge r}P(\sqrt{\kappa_2}B(s)>q-x)P(\sqrt{\kappa_2}B(t)>r-x)dx\\&-\mathbf{1}\{r>q\}\int_q^rP(\sqrt{\kappa_2}B(s)\leq q-x)P(\sqrt{\kappa_2}B(t)>r-x)dx\\&-\mathbf{1}\{q>r\}\int_r^qP(\sqrt{\kappa_2}B(s)>q-x)P(\sqrt{\kappa_2}B(t)\leq r-x)dx\end{align*}
and 
\begin{align*}\Gamma_q^{(1)}((s,q),(t,r))&=\int_{-\infty}^\infty\{P(\sqrt{\kappa_2}B(s)\leq q-x)P(\sqrt{\kappa_2}B(t)>r-x)\\&-P(\sqrt{\kappa_2}B(s)\leq q-x,\sqrt{\kappa_2}B(t)>r-x)\}dx;\end{align*}
$B(\cdot)$ is standard Brownian motion.
The covariance terms (\ref{gamma0}) and (\ref{gammaq}) appear in the calculations as $\Gamma_0^{(2)}$ and $\Gamma_q^{(1)}$ resp.
\begin{lemma}
\label{s1}
$S_1$ converges to the mean-zero normal distribution with variance
\begin{align*}\sigma^2&=\sum_{1\leq i,j\leq N}\theta_i\theta_j\bigl\{\rho_0\Gamma_q^{(1)}((t_i,r_i),(t_j,r_j))+v_0\Gamma_0^{(2)}((t_i,r_i),(t_j,r_j))\bigr\}.\end{align*}\\
\end{lemma}
\begin{proof}
We apply Lindeberg-Feller to $S_1$. We check:
\begin{enumerate}
\item
For any $\epsilon >0$
\begin{equation}
\label{s1-1}
\lim_{n\to\infty}\sum_{|m|\leq [r(n)\sqrt{n}]}E[\bar{U}_m^2\mathbf{1}\{|\bar{U}_m|\geq\epsilon\}]=0
\end{equation}
and,
\item
\begin{equation}
\label{s1-2}
\lim_{n\to\infty}\sum_{|m|\leq [r(n)\sqrt{n}]}E[\bar{U}_m^2]=\sigma^2.
\end{equation}
\end{enumerate}

$|U_m(t,s)|\leq n^{-1/4}\eta_0^n(m)+n^{-1/4}\rho_0^n(m)$.
By uniform bound on moments (\ref{momentassmpn}), we get
\[|\bar{U}_m|\leq Cn^{-1/4}[\eta_0^n(m)+1].\]
By choice of $r(n)$ we get (\ref{s1-1}):\[\lim_{n\to\infty}\sum_{|m|\leq [r(n)\sqrt{n}]}E[\bar{U}_m^2{\bf 1}\{|\bar{U}_m|\geq\epsilon\}]=0.\]
We prove (\ref{s1-2}) from the lemma below.
\[\sum_{|m|\leq r(n)\sqrt{n}}E\bar{U}_m^2=\sum_{1\leq i,j\leq N}\theta_i\theta_j\sum_{|m|\leq r(n)\sqrt{n}}E[U_m(t_i,r_i)U_m(t_j,r_j)].\]
Let
\[S=\sum_{|m|\leq [r(n)\sqrt{n}]}E[U_m(s,q)U_m(t,r)].\]\\

\begin{lemma}
\begin{align*}\lim_{n\to\infty}S
=\rho_0\Gamma_q^{(1)}((s,q),(t,r))+v_0\Gamma_0^{(2)}((s,q),(t,r)).
\end{align*}
\end{lemma}
\begin{proof}
\begin{align*}&E[U_m(s,q)U_m(t,r)]\\&=n^{-1/2}\Cov\biggl[\sum_{j=1}^{\eta_0^n(m)}\bigl\{\mathbf{1} \{A^{s,q}_{m,j}\}\mathbf{1} \{m>[q\sqrt{n}]\}-\mathbf{1} \{(A^{s,q}_{m,j})^c\}\mathbf{1}\{m\leq [q\sqrt{n}]\}\bigr\}\ ,
\\
&\qquad \qquad \qquad  \sum_{j=1}^{\eta_0^n(m)}\bigl\{\mathbf{1} \{A^{t,r}_{m,j}\}\mathbf{1} \{m>[r\sqrt{n}]\}-\mathbf{1} \{(A^{t,r}_{m,j})^c\}\mathbf{1} \{m\leq [r\sqrt{n}]\}\bigr\}\biggr].\end{align*}

Assume ${z_i}$ are i.i.d. random variables independent of the random nonnegative integer $K$, then
\[\Cov[\sum_{i=1}^Kf(z_i),\sum_{i=1}^Kg(z_i)]=EK\cdot\Cov[f(z_1),g(z_1)]+\Var K\cdot Ef(z_1)\cdot Eg(z_1).\]
For fixed $m$, take $K=\eta_0^n(m)$,
\begin{equation*}f(X_{m,j})=\mathbf{1} \{A^{s,q}_{m,j}\}\mathbf{1} \{m>[q\sqrt{n}]\}-\mathbf{1} \{(A^{s,q}_{m,j})^c\}\mathbf{1} \{m\leq [q\sqrt{n}]\}\end{equation*}
and
\begin{equation*}g(X_{m,j})=\mathbf{1} \{A^{t,r}_{m,j}\}\mathbf{1} \{m>[r\sqrt{n}]\}-\mathbf{1} \{(A^{t,r}_{m,j})^c\}\mathbf{1} \{m\leq [r\sqrt{n}]\}.\end{equation*}
We now adopt the simpler notation \[A=\{X(ns)\leq [nvs]+[q\sqrt{n}]-m\}\] and \[B=\{X(nt)\leq [nvt]+[r\sqrt{n}]-m\}.\] %First we compute
A straightforward computation gives
\begin{equation*}\Cov[f(X_{m,1}),g(X_{m,1})]
=P(A)P(B^c)-P(A\cap B^c).
\end{equation*}
Next we compute
\begin{align*}Ef(X_{m,1})Eg(X_{m,1})=&P(A)P(B)\mathbf{1}\{m>[q\sqrt{n}]\}\mathbf{1}\{m>[r\sqrt{n}\}\\&+P(A^c)P(B^c)\mathbf{1}\{m\leq [q\sqrt{n}]\}\mathbf{1}\{m\leq [r\sqrt{n}\}\\&-P(A)P(B^c)\mathbf{1}\{m>[q\sqrt{n}]\}\mathbf{1}\{m\leq [r\sqrt{n}\}\\&-P(A^c)P(B)\mathbf{1}\{m\leq [q\sqrt{n}]\}\mathbf{1}\{m>[r\sqrt{n}\}.\end{align*}

Putting these computations together we get 
\begin{align*}S=&n^{-1/2}\sum_{|m|\leq [r(n)\sqrt{n}]}\biggl\{\rho_0^n(m)\bigl\{P(A) P(B^c)-P(A\cap B^c)\bigr\}\\&+v_0^n(m)\biggl[P(A)P(B)\mathbf{1}\{m>[q\sqrt{n}]\vee[r\sqrt{n}]\}+P(A^c)P(B^c)\mathbf{1}\{m\leq [q\sqrt{n}]\wedge[r\sqrt{n}\}\\&-P(A)P(B^c)\mathbf{1}\{[q\sqrt{n}]<m\leq [r\sqrt{n}\}-P(A^c)P(B)\mathbf{1}\{[r\sqrt{n}]<m\leq [q\sqrt{n}]\}\biggr]\biggr\}.\end{align*}
$X(nt)$ is a sum of Poisson($nt$) number of independent jumps $\xi_i$, each jump distributed according to $p(x)$.
Therefore, \[\Var(X(nt))=nt\cdot \Var(\xi_1)+(E\xi_1)^2\cdot (nt)=nt(\kappa_2-v^2)+v^2nt=nt\kappa_2.\] By Donsker's Invariance Principle the process $\{(X(nt)-[nvt])/\sqrt{n\kappa_2}:t\geq0\}$ converges weakly to standard 1-dimensional Brownian motion.
Therefore, as $n\to\infty$, assumption \eqref{initial_cond} and a Riemann sum argument together with the large deviation bounds on $X(nt)$ gives us
\begin{align*}&\lim_{n\to\infty}S=\rho_0\int_{-\infty}^\infty\{P(B(\kappa_2s)\leq q-x)P(B(\kappa_2t)>r-x)\\
&\qquad-P(B(\kappa_2s)\leq q-x,B(\kappa_2t)>r-x)\}dx\\
&+v_0\biggl\{\int_{q\vee r}^\infty P(B(\kappa_2s)\leq q-x)P(B(\kappa_2t)\leq r-x)dx\\
&\qquad+\int_{-\infty}^{q\wedge r}P(B(\kappa_2s)> q-x)P(B(\kappa_2t)>r-x)dx\\
&\qquad-\int_{q\wedge r}^{q\vee r}\bigl[{\bf 1}\{q<r\}P(B(\kappa_2s)\leq q-x)P(B(\kappa_2t)>r-x)\\
&\qquad +{\bf 1}\{r<q\}P(B(\kappa_2s)> q-x)P(B(\kappa_2t)\leq r-x)\bigr]\biggr\}\\
&=\rho_0\Gamma_q^{(1)}((s,q),(t,r))+v_0\Gamma_0^{(2)}((s,q),(t,r)).\end{align*}
The reasoning behind the convergence of $S$ to the above limit is the same as in Lemmas 4.3 and 4.4 of \cite{SEPP}. The reader is referred to \cite{SEPP} (page 778) for a more detailed explanation.

\end{proof}

This verifies (\ref{s1-2}):
\begin{align*}\lim_{n\to\infty}\sum_{|m|\leq r(n)\sqrt{n}}E[\bar{U}_m^2]=\sum_{1\leq i,j\leq N}\theta_i\theta_j\bigl\{\rho_0\Gamma_q^{(1)}((t_i,r_i),(t_j,r_j))+v_0\Gamma_0^{(2)}((t_i,r_i),(t_j,r_j))\bigr\}.\end{align*}

We can now conclude from Lindeberg-Feller that $S_1$ converges to mean-zero normal distribution with 
\begin{align*}\sigma^2&=\sum_{1\leq i,j\leq N}\theta_i\theta_j\bigl\{\rho_0\Gamma_q^{(1)}((t_i,r_i),(t_j,r_j))+v_0\Gamma_0^{(2)}((t_i,r_i),(t_j,r_j))\bigr\}.\end{align*}
\end{proof}

\begin{lemma}
\label{meanvanish}
\[\lim_{n\to\infty} \sup_{0\leq t\leq T,0\leq |r|\leq S} n^{-1/4}|EY_n(t,r)|=0.\]
\end{lemma}
\begin{proof}
The proof is similar to the proof of Lemma 4.6 (page 781-783) in \cite{SEPP}.
\end{proof}

Showing that 
\begin{align*}
\Gamma_q^{(1)}((s,q),(t,r))=\Psi_{\kappa_2(t+s)}(|q-r|)-
\Psi_{\kappa_2(|s-t|)}(|q-r|)=\Gamma_q((s,q),(t,r))
\end{align*}
and
$$\Gamma_0^{(2)}((s,q),(t,r))=\Gamma_0((s,q),(t,r))$$
is an exercise in calculus.

\begin{proof}[{\bf Proof of Proposition \ref{finite_dim}}]
By Lemmas \ref{s2}, \ref{s1} and \ref{meanvanish} we have, as $n\to\infty$, $n^{-1/4}\sum_{i=1}^N\theta_iY_n(t_i,r_i)$ converges to a mean-zero normal distribution with variance $\sigma^2$,
where \[\sigma^2=\sum_{1\leq i,j\leq N}\theta_i\theta_j\{\rho_0\Gamma_q((t_i,r_i),(t_j,r_j))+v_0\Gamma_0((t_i,r_i),(t_j,r_j))\},\]  for arbitrary $(\theta_1,...,\theta_N)\in \mathbf R^N$. We can therefore conclude that as $n\to\infty$ the vector $n^{-1/4}(Y_n(t_1,r_1),Y_n(t_2,r_2),...,Y_n(t_N,r_N))$ converges in distribution to the mean-zero Gaussian random vector $(Z(t_1,r_1),Z(t_2,r_2),...,Z(t_N,r_N))$ with covariance 
\[EZ(t_i,r_i)Z(t_j,r_j)=\rho_0\Gamma_q((t_i,r_i),(t_j,r_j))+v_0\Gamma_0((t_i,r_i),(t_j,r_j)).\]

\end{proof}

\numberwithin{equation}{section}
\section{\bf Tightness and completion of proof of Theorem 2.1}\label{sec:tightness}
In this section we first develop a criterion for tightness for processes in $D_2$. The tightness criterion is in terms of a modulus of continuity. We then proceed to check if our scaled current process satisfies the tightness criterion.
The following is an extension of Proposition 5.7 in \cite{Durr} to two-parameter processes. WLOG for simplicity we replace the region $[0,T]\times[-S,S]$ with the unit square $[0,1]^2$.  For any $h\in D_2=D_2([0,1]^2,\mathbb R)$, define \[w_h(\delta)=\sup_{\substack{s,t,q,r\in[0,1]\\|(s,q)-(t,r)|<\delta}}|h(s,q)-h(t,r)|.\]
\begin{prop}
\label{criterion}
Suppose \{$X_n$\} is a sequence of random elements of $D_2=D_2([0,1]^2,\mathbb R)$ satisfying these conditions:\\
For all n there exists $\delta_n>0$ such that
\begin{enumerate}
  \item there exist $\beta>0$, $\sigma>2$ , and $C>0$ such that for all n sufficiently large
\begin{equation}
\label{momentbound}
 E(|X_n(s,q)-X_n(t,r)|^\beta)\leq C|(s,q)-(t,r)|^\sigma
\end{equation}
 for all $s,t,q,r\in [0,1]$ with $|(s,q)-(t,r)|>\delta_n$, and
  \item for every $\epsilon>0$ and $\eta>0$ there exists an $n_0$ such that
\begin{equation}
\label{tightmodofcont}
  P(w_{X_n}(\delta_n)>\epsilon)<\eta  \text{  for all } n\geq n_0.
\end{equation}
\end{enumerate}
Then, for each $\epsilon>0$ and $\eta>0$, there exists a $\delta$, $0<\delta<1$, and an integer $n_0$, such that \[P(w_{X_n}(\delta)\geq \epsilon)\leq \eta, \text{ for  }n\geq n_0.\]
\end{prop}

To prove this proposition we require the following lemma.

\begin{lemma}
\label{sequence}
Let $0\leq k_0\leq k$. If points $(s,q)$ and $(t,r)$ lie on the $2^{-k}$ grid i.e. $s=\frac{i}{2^k}$, $q=\frac{j}{2^k}$, $t=\frac{i'}{2^k}$, $r=\frac{j'}{2^k}$, and  $|\frac{i}{2^k}-\frac{i'}{2^k}|\leq 2^{-k_0}$, $|\frac{j}{2^k}-\frac{j'}{2^k}|\leq 2^{-k_0}$, then 
\begin{enumerate}
\item
it is possible to move from $(s,q)$ to $(t,r)$ in steps of size $2^{-h}$, $k_0\leq h\leq k$, moving one co-ordinate at a time, where a step of size $2^{-h}$, for any h, occurs at most 4 times;
\item
also, we can choose our steps in such a way that we make a jump of size $2^{-h}$ only if we lie in the $2^{-h}$ grid.
\end{enumerate}
\end{lemma}
\begin{proof}
We fix $k_0$ and prove the lemma by induction on $k$. When $k=k_0$,
$(s,q)$, $(t,r)$ $\in 2^{-k_0}\mathbb N\times 2^{-k_0}\mathbb N$.
We are given that  $|\frac{i}{2^{k_0}}-\frac{i'}{2^{k_0}}|\leq 2^{-k_0}$, $|\frac{j}{2^{k_0}}-\frac{j'}{2^{k_0}}|\leq 2^{-k_0}.$
So, either both points coincide or there is a difference of $2^{-k_0}$ in one co-ordinate or both co-ordinates between these points.
We can therefore move from $(s,q)$ to $(t,r)$ in at most 2 steps, each of size $2^{-k_0}$, moving one co-ordinate at a time.
Clearly condition (2) also holds as the only jumps possible here are of size $2^{-k_0}$ and we lie in the $2^{-k_0}$ grid.

Let $k>k_0$,
$(s,q)=(\frac{i}{2^k},\frac{j}{2^k})$, $(t,r)=(\frac{i'}{2^k},\frac{j'}{2^k})$ ,$|\
\frac{i}{2^k}-\frac{i'}{2^k}|\leq 2^{-k_0}$ and $|\frac{j}{2^k}-\frac{j'}{2^k}|\leq 2^{-k_0}$. 
Either the points $(s,q)$ and $(t,r)$ already lie on the $2^{-(k-1)}$ grid, or if not, the points $(s,q)$ and $(t,r)$ are each at most two jumps of size $2^{-k}$ away from the $2^{-(k-1)}$-grid. We can therefore move both $(s,q)$ and $(t,r)$ closer to each other and onto the $2^{-(k-1)}$-grid in at most four jumps of size $2^{-k}$.

By the induction hypothesis and by our choice of jumps, we are done.
\end{proof}

\begin{proof}[{\bf Proof of Proposition \ref{criterion}}]
We may assume without loss of generality that $\delta_n=2^{-k}$ for some $k=k(n)\geq 0$ depending on $n$. This is because if we have $P(w_{X_n}(\delta_n)>\epsilon)<\eta$ for some $\delta_n>0$, then we can find a $k(n)>0$ with $2^{-k(n)-1}<\delta_n<2^{-k(n)}$ such that $w_{X_n}(2^{-k(n)})\leq 2w_{X_n}(\delta_n)$. So $P(w_{X_n}(2^{-k(n)})>2\epsilon)\leq P(w_{X_n}(\delta_n)>\epsilon)<\eta$. Therefore it is sufficient to prove the theorem for $\delta_n=2^{-k(n)}$ for $k(n)\geq 0$.

Given $n$ and $\delta$, if $\delta\leq \delta_n$, then $w_{X_n}(\delta)\leq w_{X_n}(\delta_n)$. 
If $\delta_n<\delta$, then for any two points $(s,q), (t,r)$ with $0\leq |t-s|,|q-r|\leq \delta$, we have
\[|h(s,q)-h(t,r)|\leq |h(s,q)-h(s',q')|+|h(t,r)-h(t',r')|+|h(s',q')-h(t',r')|,\]
where $s',t',q',r'\in \delta_n\mathbb N$, $0\leq |s'-t'|,|q'-r'|\leq \delta$ and $0\leq |t-t'|,|s-s'|,|r-r'|,|q-q'|\leq \delta_n$. Thus for any $\delta_n<\delta$,
\[ w_h(\delta)\leq2w_h(\delta_n)+\sup_{\substack{t,s,q,r\in (\delta_n\mathbb N)\bigcap [0,1]\\ |t-s|,|r-q|\leq \delta}}|h(s,q)-h(t,r)|.\]
 Therefore by (\ref{tightmodofcont}), we only need to show that for any $\epsilon>0$ and $\eta>0$, there exists a $\delta>0$ such that for all $n$ sufficiently large
 \begin{equation}
 \label{modofcont-n}
  P(w_{X_n}^{(n)}(\delta)>\epsilon)<\eta,
 \end{equation}
where $$w_h^{(n)}(\delta)=\sup_{\substack{t,s,q,r\in(\delta_n\mathbb N)\bigcap [0,1]\\ |t-s|,|r-q|\leq \delta}}|h(s,q)-h(t,r)|.$$
This follows from (\ref{momentbound}) by the following ``dyadic argument'':\\
Choose a $\lambda$ such that $2^{(2-\sigma)}<\lambda^\beta<1.$ Given $n$ for which (\ref{momentbound}) is satisfied, let
\begin{align*}G_{k}=\bigl\{&\bigl|X_n(\tfrac{i}{2^k},\tfrac{j}{2^{k}})-X_n(\tfrac{i+1}{2^{k}},\tfrac{j}{2^{k}})\bigr|\leq \lambda^{k} \text{ for }i=0,1,\ldots,2^{k}-1,\\&j=0,1,\ldots,2^{k}\bigr\}\end{align*}
and
\begin{align*}H_{k}=\bigl\{&\bigl|X_n(\tfrac{i}{2^k},\tfrac{j}{2^{k}})-X_n(\tfrac{i}{2^{k}},\tfrac{j+1}{2^{k}})\bigr|\leq \lambda^{k} \text{ for }i=0,1,\ldots,2^{k},\\&j=0,1,\ldots,2^{k}-1\bigr\}\end{align*}
where $k\leq k(n)$ .
\begin{align*}
P(G_{k}^c\cup H_{k}^c)&\leq 2^{k}(2^{k}+1)\lambda^{-k\beta}2^{-k\sigma}+2^{k}(2^{k}+1)\lambda^{-k\beta}2^{-k\sigma}
&\intertext{by Markov inequality and (\ref{momentbound}) as $k\leq k(n)$} 
&\leq c(2^{(2-\sigma)}\lambda^{-\beta})^{k}\\
&=c\gamma^{k}\text{   where }\gamma=2^{(2-\sigma)}\lambda^{-\beta}<1.
\end{align*}

Given $\epsilon>0$ and $\eta>0$, choose $k_0$ such that\\
\[ c\sum_{k\geq k_0}\gamma^{k}<\eta \text{ and }
4\sum_{k\geq k_0}\lambda^{k}<\epsilon.\]
Choose $\delta=2^{-k_0}$.
If $\delta<\delta_n$ for some $n\geq n_0$, then $w_{X_n}(\delta)<w_{X_n}(\delta_n)$ and we have $P(w_{X_n}(\delta)>\epsilon)<\eta$ by (\ref{tightmodofcont}).
A little more work is required to show that (\ref{modofcont-n}) holds in the $\delta_n\leq \delta$, i.e. $k_0\leq k(n)$, case.
Pick any $(s,q),(t,r)\in \delta_n\mathbb N\times\delta_n\mathbb N$ where $\delta_n=2^{-k(n)}$, such that $|s-t|,|q-r|<\delta$.
 We can find a sequence of points $(s,q)=(s_1,q_1),(s_2,q_2)...(s_m,q_m)=(t,r)$ (refer to lemma \ref{sequence}) such that on the event $\bigcap_{k_0\leq k\leq k(n)}(G_{k}\cap H_{k})$
we have
\[|X_n(s,q)-X_n(t,r)|\leq\sum_{i=1}^{m-1}|X_n(s_i,q_i)-X_n(s_{i+1},q_{i+1})|\leq 4\sum_{k\geq k_0}\lambda^{k}<\epsilon.\]
Now
\begin{equation*}
P\bigl(\bigcup_{k_0\leq k\leq k(n)}(G_{k}^c\cup H_{k}^c)\bigr)\leq c\sum_{k_0\leq k\leq k(n)}\gamma^{k}<\eta .\end{equation*}
 Therefore,
\[ P(w_{X_n}^{(n)}(\delta)\leq \epsilon)\geq P(\bigcap_{k_0\leq k\leq k(n)}(G_k\cap H_k))\geq 1-\eta .\]
Thus (\ref{modofcont-n}) is satisfied with $\delta=2^{-k_0}$.
\end{proof}

We now apply Proposition \ref{criterion} to processes $X_n=n^{-1/4}\bar{Y}_n(t,r)$ where $\bar{Y}_n(t,r)=Y_n(t,r)-EY_n(t,r)$. 
\subsection{Verifying the first tightness condition} 
We check that (\ref{momentbound}) holds for $n^{-1/4}\bar{Y}_n(t,r)$.
Let $\alpha>0$ and  \begin{equation}
\label{alpha}
5/4+\alpha<\beta<3/2.
\end{equation}
We show that there exist constants $\sigma>2$ and $0<C<\infty$ independent of $n$, such that with $\delta_n=n^{-\beta}$, for all $n$ (sufficiently large)
\begin{equation}
E\left(n^{-1/4}|\bar Y_n(s,q)-\bar{Y}_n(t,r)|\right)^{12}\leq C|(s,q)-(t,r)|^\sigma
\end{equation}
for all $(s,q),(t,r)\in [0,T]\times [-S,S] $ with $|(s,q)-(t,r)|>\delta_n$.

We can assume WLOG that $s\leq t$ in the following calculations.
Define \[A_{m,j}=\left\{ X_{m,j}(ns)\leq [q\sqrt{n}\ ]+[nvs],X_{m,j}(nt)>[r\sqrt{n}\ ]+[nvt] \right\}\] and 
\[B_{m,j}=\left\{X_{m,j}(ns)> [q\sqrt{n}\ ]+[nvs],X_{m,j}(nt)\leq [r\sqrt{n}\ ]+[nvt]\right\}.\]
If $q\leq r$ then
\begin{align}
Y_n(s,q)-Y_n(t,r)=&\sum_{m>[r\sqrt{n}]}\sum_{j=1}^{\eta_0^n(m)}{\bf 1}\{X_{m,j}(ns)\leq [q\sqrt{n}]+[nvs]\}\label{e1}\\
&+\sum_{m=[q\sqrt{n}]+1}^{[r\sqrt{n}]}\sum_{j=1}^{\eta_0^n(m)}{\bf 1}\{X_{m,j}(ns)\leq [q\sqrt{n}]+[nvs]\}\label{e2}\\
&-\sum_{m\leq [q\sqrt{n}]}\sum_{j=1}^{\eta_0^n(m)}{\bf 1}\{X_{m,j}(ns)>[q\sqrt{n}]+[nvs]\}\label{e3}\\
&-\sum_{m>[r\sqrt{n}]}\sum_{j=1}^{\eta_0^n(m)}{\bf 1}\{X_{m,j}(nt)\leq [r\sqrt{n}]+[nvt]\}\label{e4}\\
&+\sum_{m=[q\sqrt{n}]+1}^{[r\sqrt{n}]}\sum_{j=1}^{\eta_0^n(m)}{\bf 1}\{X_{m,j}(nt)>[r\sqrt{n}]+[nvt]\}\label{e5}\\
&+\sum_{m\leq [q\sqrt{n}]}\sum_{j=1}^{\eta_0^n(m)}{\bf 1}\{X_{m,j}(nt)>[r\sqrt{n}]+[nvt]\}\label{e6}
\end{align}
Combining (\ref{e1}) and (\ref{e4}), (\ref{e2}) and (\ref{e5}), (\ref{e3}) and (\ref{e6}) and adding and subtracting 
\[\sum_{m=[q\sqrt{n}]+1}^{[r\sqrt{n}]}\sum_{j=1}^{\eta_0^n(m)}{\bf 1}\{X_{m,j}(ns)>[q\sqrt{n}]+[nvs],X_{m,j}(nt)\leq [r\sqrt{n}]+[nvt]\}\] we get
\begin{equation*}
\bar{Y}_n(s,q)-\bar{Y}_n(t,r)=\sum_{m\in\mathbb Z }G_m+\sum_{m=[q\sqrt{n}\ ]+1}^{[r\sqrt{n}\ ]}\bigl[\eta_0^n(m)-\rho_0^n(m)\bigr]
\end{equation*}
where
\begin{equation}
G_m=\sum_{j=0}^{\eta_0^n(m)}({\bf 1}_{A_{m,j}}-{\bf 1}_{B_{m,j}})-\rho_0^n(m)( P(A_{m,1})- P(B_{m,1})).
\end{equation}
Similarly, when $q>r$
\begin{equation*}
\bar{Y}_n(s,q)-\bar{Y}_n(t,r)=\sum_{m\in\mathbb Z }G_m-\sum_{m=[r\sqrt{n}\ ]+1}^{[q\sqrt{n}\ ]}\bigl[\eta_0^n(m)-\rho_0^n(m)\bigr].
\end{equation*}
Using the identity $(a+b)^k\leq 2^k(a^k+b^k)$, we get
\begin{equation}
\label{b}
 E\left(n^{-1/4}(\bar{Y}_n(s,q)-\bar{Y}_n(t,r))\right)^{12}\leq \frac{2^{12}}{n^3}\left( EA^{12}+ EB^{12}\right)
\end{equation}
where
\[A=\sum_{m\in\mathbb Z }G_m\]
and \[B=\sum_{m=[(r\wedge q)\sqrt{n}\ ]+1}^{[(r\vee q)\sqrt{n}\ ]}\bigl[\eta_0^n(m)-\rho_0^n(m)\bigr].\]
We now bound $E[A^{12}]$ and $E[B^{12}]$.

To bound $EB^{12}$, we use the following lemma which is a slight modification of Lemma 8 in \cite{MS}.
\begin{lemma}
\label{eta}
Let $Y_i$ be independent random variables with $ E[|Y_i|^{2r}]<c<\infty$ and $ EY_i=0$ for all $i$ and for some fixed $r>0$. There is a constant $C<\infty$ such that, for any n ,
 \[ E\biggl\{\bigl(Y_1+Y_2+\cdots+Y_n\bigr)^{2r}\biggr\}\leq C(2r)!n^r.\]
\end{lemma}
\begin{proof}
Since $EY_i=0$,
\[E\biggl\{\biggl(Y_1+Y_2+\cdots+Y_n\biggr)^{2r}\biggr\}={\sum}'\frac{(2r)!}{r_1!r_2!\cdots r_n!} EY_1^{r_1} EY_2^{r_2}\cdots EY_n^{r_n}\]
where $\sum '$ extends over all n-tuples of integers $r_1,r_2,\hdots r_n\geq 0$ such that each $r_i\neq 1$ and $r_1+\cdots +r_n=2r$.\\
$| EY_1^{r_1} EY_2^{r_2}\cdots EY_n^{r_n}|\leq c$ by the bounded moment assumption and H\"older's inequality.
The number $A=\sum '(r_1!\cdots r_n!)^{-1}$ is the coefficient of $x^{2r}$ in
$$F_n(x)=\bigl(\sum_{\substack{j\geq 0\\ j\neq 1}}\frac{x^j}{j!}\bigr)^n$$
and consequently $A\leq \frac{F_n(x)}{x^{2r}}$ for $x>0$. But
$\sum_{\substack{j\geq 0\\ j\neq 1}}\frac{x^j}{j!}\leq 1+x^2$
for $x\leq 1$, so taking $x=n^{-1/2}$ gives $A\leq n^r(1+(1/n))^n\leq Cn^r$.
\end{proof}

Recall that $\eta_0^n(x),x\in \mathbb Z$ are independent with mean $\rho_0^n(x)$. Let $C$ denote a constant that varies from line to line in the string of inequalities below.
Applying Lemma \ref{eta} to the term $E[B^{12}]$ in equation (\ref{b}) with $r=6$ and by the moment assumption on $\eta_0^n$, we get
\begin{equation}
\label{b1}
 E[B^{12}] \leq C(12)!|[r\sqrt{n}]-[q\sqrt{n}]|^6
 \leq C(|r-q|^6n^3+1).
\end{equation}

We use the following lemma to bound $E[A^{12}]$.
\begin{lemma}
There exists a constant C such that for each positive integer $1\leq k\leq 12$ and for all m,
$$ E[|G_m|^k]\leq D_m$$
where $D_m=C\{P(A_{m,1})+ P(B_{m,1})\}$.
\end{lemma}
\begin{proof}
The proof is the same as in Lemma 4.7 (pages 784-785) of \cite{SEPP}.
\end{proof}
\begin{lemma}
There exists a $c>0$ such that
\begin{equation}
\label{b2}
 E[A^{12}]\leq c\left\{1+\left(\sum_mD_m\right)^6\right\}.
\end{equation}
\end{lemma}
\begin{proof}
$A=\sum_{m\in\mathbb Z }G_m$. Note that $ EG_m=0$ for all $m$ and the $G_m$'s are independent. Let ${\sum}^{(k)}$ denote the sum over all $k$-tuples of integers $r_1,r_2,\hdots,r_{k}\geq 2$ such that  $r_1+\cdots+r_k=12$.
\begin{equation}
\begin{split}
 E[A^{12}]&=\sum_{m_1,\hdots,m_{12}\in \mathbb Z} E[G_{m_1}G_{m_2}\cdots G_{m_{12}}]\\
&\leq \sum_{k=1}^6{\sum}^{(k)}\frac{12!}{r_1!\cdots r_{k}!}\sum_{m_1\neq m_2\neq\cdots m_{k}} E|G_{m_1}|^{r_1} E|G_{m_2}|^{r_2}\cdots E|G_{m_{k}}|^{r_{k}}\\
&\leq \sum_{k=1}^6{\sum}^{(k)}\frac{12!}{r_1!r_2!\cdots r_{k}!}\sum_m E|G_m|^{r_1}\sum_m E|G_m|^{r_2}\cdots\sum_m E|G_m|^{r_{k}}.
\end{split}
\end{equation}
Since $ E|G_m|^l\leq D_m$ for all $m\in\mathbb Z$ and $1\leq l\leq 12$
we get, for all $1\leq k\leq 6$,
\begin{align*}
\sum_m E|G_m|^{r_1}\sum_m E|G_m|^{r_2}\cdots \sum_m E|G_m|^{r_k}&\leq \max\{1,(\sum_mD_m)^6\}
\leq \{1+(\sum_mD_m)^6\}.
\end{align*}
Thus 
\begin{equation}
\label{b6}
 E[A^{12}]\leq c\left\{1+\left(\sum_mD_m\right)^6\right\}.
\end{equation}
\end{proof} 
We evaluate $\sum_{m\in\mathbb Z}D_m$ below. Recall that
\[\sum_{m\in\mathbb Z} D_m= C\left(\sum_{m\in\mathbb Z}P(A_{m,1})+\sum_{m\in\mathbb Z} P(B_{m,1})\right).\]
\begin{equation}
\label{b3}
\begin{split}
\sum_{m\in\mathbb Z} P(A_{m,1})&=\sum_{m\in\mathbb Z} P(X(ns)\leq [q\sqrt{n}]+[nvs]-m,X(nt)>[r\sqrt{n}]+[nvt]-m)\\
&=\sum_{m\in\mathbb Z}\sum_{l\geq m}\{ P(X(ns)=[q\sqrt{n}]+[nvs]-l)\\
&\qquad \times P(X(nt)-X(ns)>[nvt]-[nvs]+[r\sqrt{n}]-[q\sqrt{n}]+l-m)\}\\
&=\sum_{l\in\mathbb Z}\sum_{k\geq 0}\{ P(X(ns)=[q\sqrt{n}]+[nvs]-l)\\
&\qquad\times P(X(n(t-s))>[nvt]-[nvs]+[r\sqrt{n}]-[q\sqrt{n}]+k)\}\\
&\leq \sum_{k\geq 0} P(X(n(t-s))-[nv(t-s)]>[r\sqrt{n}]-[q\sqrt{n}]+k)
\end{split}
\end{equation}
Similarly,
\begin{equation}
\sum_{m\in\mathbb Z} P(B_{m,1})\leq \sum_{k<0}P(X(n(t-s))-[nv(t-s)]\leq [r\sqrt{n}]-[q\sqrt{n}]+k+1)
\end{equation}
Together,
\begin{align*}
\sum_{m\in\mathbb Z} P(A_{m,1})+\sum_{m\in\mathbb Z} P(B_{m,1})&\leq  E\bigl|X(n(t-s))-[nv(t-s)]+[q\sqrt{n}]-[r\sqrt{n}]\bigr|\\ 
&\leq  E\bigl|X(n(t-s))-[nv(t-s)]\bigr|+|r-q|\sqrt{n}+1\\
&\leq c\{\sqrt{n(t-s)}+|r-q|\sqrt{n}+1\}.\\
\end{align*} 
Consequently, (\ref{b2}) becomes 
\begin{equation} 
\label{A12} 
E[A^{12}]\leq c\{1+(n(t-s))^3+|r-q|^6n^3\}.
\end{equation} 
Putting (\ref{b1}) and (\ref{A12}) together in (\ref{b}) we get
\begin{equation*}
 E[\{n^{-1/4}(\bar{Y}_n(s,q)-\bar{Y}_n(t,r))\}^{12}]\leq c\bigl(n^{-3}+(t-s)^3+|r-q|^6\bigr)
\end{equation*}
Using $\beta<3/2$, $|r-q|\leq 2S<\infty$ and $t-s\leq T<\infty$, we can find constants $c>0$ and $\sigma>2$ such that 
\[
E[\{n^{-1/4}(\bar{Y}_n(s,q)-\bar{Y}_n(t,r))\}^{12}\ ] \leq c(\ |r-q|^\sigma+|t-s|^\sigma),\]
if $|r-q|>n^{-\beta}$ or $t-s>n^{-\beta}$.
This verifies the first tightness condition (\ref{momentbound}).

\subsection{Verifying the second tightness condition}
To verify (\ref{tightmodofcont}) for $n^{-1/4}\bar{Y}_n(t,r)$, it is sufficient to show 
\begin{lemma} \label{tight-mod}
For any $0<T,S<\infty$ and $\epsilon>0$, 
\begin{equation} 
\label{2}   
\begin{split}\lim_{n\to\infty}P\biggl\{\bigcup_{\substack{
0\leq k_1\leq [Tn^\beta]\\
[-n^\beta S]\leq k_2 \leq [n^\beta S]}}
\biggl[\sup_{\substack{
k_1n^{-\beta}\leq t\leq (k_1+1)n^{-\beta}\\
k_2n^{-\beta}\leq r\leq (k_2+1)n^{-\beta}}}
|Y_n(t,r)-Y_n(n^{-\beta}k_1,n^{-\beta}k_2)|\geq n^{1/4}\epsilon\biggr]\biggr\}=0. \end{split} 
\end{equation}   
\end{lemma} 
\begin{proof} 
 Recall from \eqref{alpha} that $5/4+\alpha<\beta<3/2$, $\alpha>0$.  We first show that particles starting at a distance of $n^{1/2+\alpha}$ or more from the interval $([-(S+1)\sqrt{n}],[S\sqrt{n}])$ do not contribute to $Y_n(\cdot,r)$ during time interval $[0,T]$, $r\in[-S,S]$, in the $n\to\infty$ limit.
\begin{lemma}
Let 
\label{lemman1}
\begin{equation}\label{n1}\begin{split}
N_1=&\sum_{m\leq [-(S+1)\sqrt{n}]-n^{1/2+\alpha}}\sum_{j=1}^{\eta_0^n(m)}{\bf 1} \{X_{m,j}(nt)\geq-[(S+1)\sqrt{n}]+[nvt]\\
&\qquad \qquad \qquad \qquad \qquad \qquad \qquad          
\text{for some }0\leq t\leq T\}\\
&+\sum_{m\geq [S\sqrt{n}]+n^{1/2+\alpha}}\sum_{j=1}^{\eta_0^n(m)}{\bf 1} \{X_{m,j}(nt)\leq[S\sqrt{n}]+[nvt]\\
&\qquad \qquad \qquad \qquad \qquad \qquad \qquad 
\text{ for some } 0\leq t\leq T\}.
\end{split}
\end{equation}
Then $EN_1\to 0$ as $n\to \infty$.
\end{lemma}
\begin{proof}
Choose a positive integer $M$ large enough so that $1/2-\alpha(2M-1)<0$.
 The expectation of the first sum in (\ref{n1}) is bounded by 
\begin{align*}
C&\sum_{m\leq[-(S+1)\sqrt{n}]-n^{1/2+\alpha}}P(\sup_{0\leq t\leq T}(X(nt)-nvt)\geq -[(S+1)\sqrt{n}]-1-m)\\
&\leq C\sum_{l\geq n^{1/2+\alpha}}P(\sup_{0\leq t\leq T}(X(nt)-nvt)\geq l-1)\\
&\leq C\sum_{l\geq n^{1/2+\alpha}}l^{-2M}E[(X(nT)-nvT)_+^{2M}] 
\intertext{ by application of Doob's inequality to the martingale  $X(t)-vt$}
&\leq C\sum_{l\geq n^{1/2+\alpha}}l^{-2M}n^M
\intertext{ as $E[(X(nT)-nvT)_+^{2M}]$ is O($n^M$)}
&\leq Cn^{1/2-\alpha(2M-1)}\to 0 \text{ as }n\to\infty.
\end{align*}
Similarly the expectation of the other sum goes to 0 as $n\to\infty$.
\end{proof}
Let \[N_2=\sum_{\substack{m=[-(S+1)\sqrt{n}]- n^{1/2+\alpha};\\m\in\mathbb Z}}^{([S\sqrt{n}])+n^{1/2+\alpha}}\eta_0^n(m).\]
be the number of particles initially within $n^{1/2+\alpha}$ distance of the interval $(-[(S+1)\sqrt{n}],[S\sqrt{n}])$.  
Fix a constant $c$ so that \begin{equation}\label{N_2}\lim_{n\to\infty} P(N_2\geq cn^{1/2+\alpha})=0.\end{equation}
Consider the event 
\begin{equation}
\label{event}
\bigcup_{\substack{
0\leq k_1\leq [Tn^\beta]\\
[-n^\beta S]\leq k_2 \leq [n^\beta S]}}
\biggl\{\sup_{\substack{
k_1n^{-\beta}\leq t\leq (k_1+1)n^{-\beta}\\
k_2n^{-\beta}\leq r\leq (k_2+1)n^{-\beta}}}|Y_n(t,r)-Y_n(n^{-\beta}k_1,n^{-\beta}k_2)|\geq n^{1/4}\epsilon\biggr\}.
\end{equation}
If $t_0=n^{-\beta}k_1$ and $r_0=n^{-\beta}k_2$, then
$$|Y_n(t,r)-Y_n(t_0,r_0)|\leq |Y_n(t,r)-Y_n(t_0,r)|+|Y_n(t_0,r)-Y_n(t_0,r_0)|.$$\\
For fixed $k_1$ and $k_2$, the event in braces in (\ref{event}) is contained in the following union of two events:
\begin{subequations} 
\begin{align}
&\left\{\sup_{t_0\leq t\leq t_0+n^{-\beta}}\sup_{r_0\leq r\leq r_0+n^{-\beta}}|Y_n(t,r)-Y_n(t_0,r)|\geq \frac{1}{2}\epsilon n^{1/4}\right\}\label{event1}\\
&\bigcup\left\{\sup_{r_0\leq r\leq r_0+n^{-\beta}}|Y_n(t_0,r)-Y_n(t_0,r_0)|\geq \frac{1}{2}\epsilon n^{1/4}\right\}\label{event2}
\end{align}
\end{subequations}
The first event (\ref{event1}) implies that at least one of the following two things happen:\\
\begin{figure}[t]
\centering
\includegraphics[width=4in]{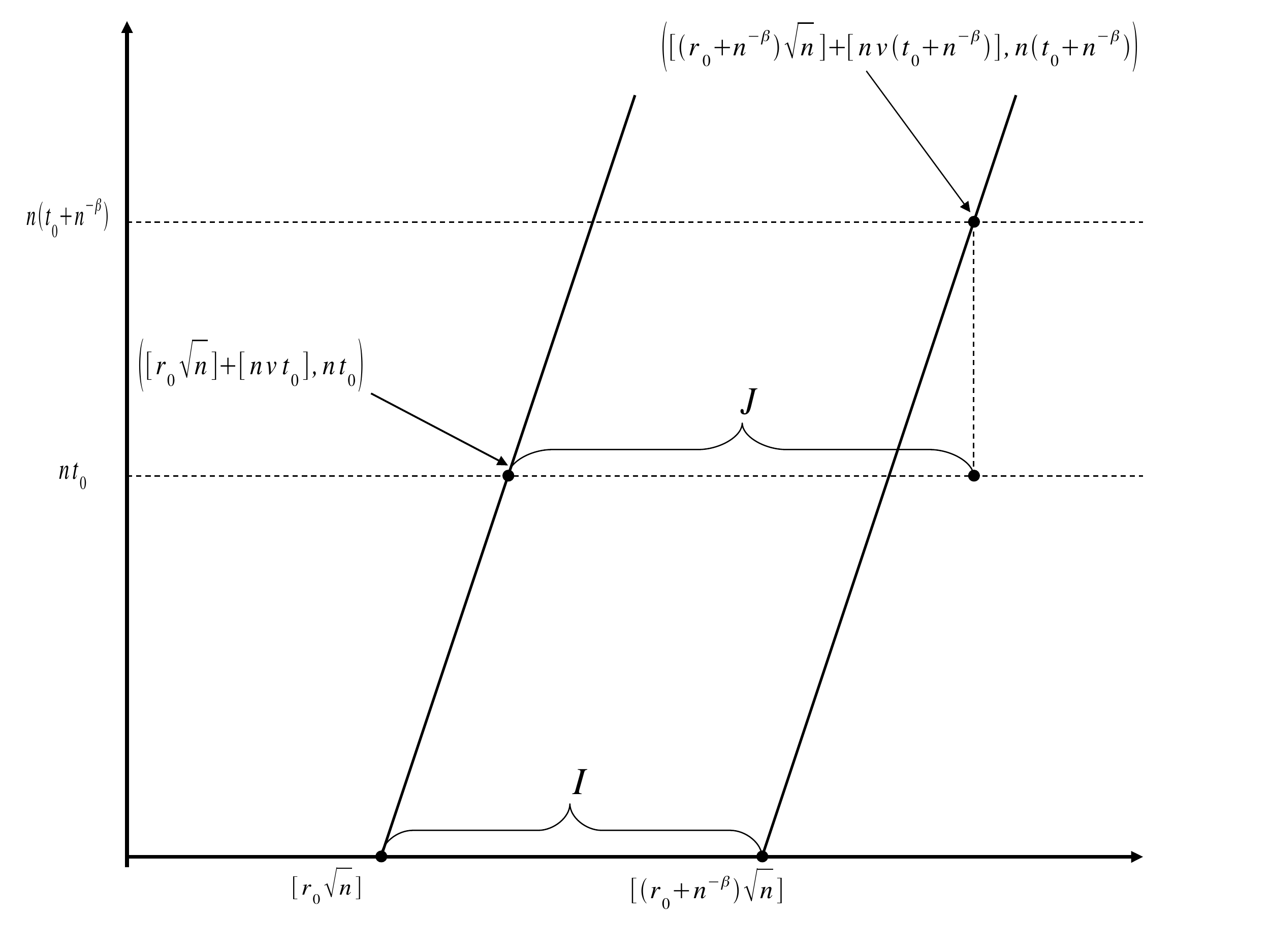}
\caption{Two characteristic lines at distance $\delta_n\sqrt{n}=n^{-\beta}\sqrt{n}$ apart.}\label{fig}
\end{figure}
\begin{enumerate}
\item
At least $\frac{1}{4}\epsilon n^{1/4}$ particles cross the discretized characteristic \[s\mapsto [r\sqrt{n}]+[nvs]\] for some $r\in [r_0,r_0+n^{-\beta}]$, during time interval $s\in [t_0,(t_0+n^{-\beta})]$ by jumping. On the event $\{N_1=0\}$, these particles must be among the $N_2$ particles initially within $n^{1/2+\alpha}$ distance of the interval $\left(-[(S+1)\sqrt{n}],[S\sqrt{n}]\right)$. Therefore, conditioned on $\{N_1=0\}$, the probability of this event is bounded by the probability that $N_2$ independent rate $1$ random walks altogether experience at least $\frac{1}{4}\epsilon n^{1/4}$ jumps in time interval of length $n^{1-\beta}$.
\item
At least $\frac{1}{4}\epsilon n^{1/4}$ particles cross the discretized characteristic \[s\mapsto [r\sqrt{n}]+[nvs]\] for some $r\in [r_0,r_0+n^{-\beta}]$, during time interval $ [t_0,(t_0+n^{-\beta})]$ by staying put while the characteristic crosses the location of these particles. These particles must lie in the interval $J$ at time $nt_0$, where \[J=\JJ ,\]  (refer to figure \ref{fig}). For large enough $n$, the distance between the endpoints of $J$ is at most $2$. So these $\frac{1}{4} \epsilon n^{1/4}$ particles must sit on at most $2$ sites, say $x_{k_1,k_2}^1$ and $x_{k_1,k_2}^2$.
\end{enumerate}
The second event (\ref{event2}) implies that at least $\frac{1}{2}\epsilon n^{1/4}$ particles either lie in the interval $\bigl[[r_0\sqrt{n}]+[nvt_0],[r\sqrt{n}]+[nvt_0]\bigr]$ at time $nt_0$, or lie in the interval $\bigl[[r_0\sqrt{n}],[r\sqrt{n}]\bigr]$ at time $0$. Since \[\left[[r_0\sqrt{n}]+[nvt_0],[r\sqrt{n}]+[nvt_0]\right]\subseteq J\] and  \[\left[[r_0\sqrt{n}],[r\sqrt{n}]\right]\subseteq I\] where $I=$\II (refer to figure \ref{fig}), this event implies that at least $\frac{1}{4}\epsilon n^{1/4}$ particles lie in interval $J$ at time $nt_0$ or at least $\frac{1}{4}\epsilon n^{1/4}$ particles lie in interval $I$ at time $0$.
For large enough $n$, the distance between the endpoints of $I$ is at most $1$, so the $\frac{1}{4}\epsilon n^{1/4}$ particles lying in interval $I$ at time $0$ must sit on a unique site $x_{0,k_2}^0$ say.

Let $\Pi(cn^{3/2+\alpha-\beta})$ denote a mean $cn^{3/2+\alpha-\beta}$ Poisson random variable that represents the total number of jumps among $cn^{1/2+\alpha}$ independent particles during a time interval of length $n^{1-\beta}$.
Then,
\begin{equation}
\begin{split}
&P\left(\text{ event in (\ref{2})}\right)\\
&\leq  P\left(N_1\geq 1\right)+ P\left(N_2\geq cn^{1/2+\alpha}\right)\\
&+\sum_{k_1=0}^{[Tn^\beta]}\sum_{k_2=[-Sn^\beta]}^{[Sn^\beta]}\biggl\{P\left(\Pi(cn^{3/2+\alpha-\beta})\geq \tfrac{1}{4}\epsilon n^{1/4}\right)+P\left(\eta_0^n(x_{0,k_2}^0)\geq \tfrac{1}{4}\epsilon n^{1/4}\right)\\
&+2\bigl\{ P\left(\eta_{n^{1-\beta}k_1}^n(x_{k_1,k_2}^1)\geq \tfrac{1}{8}\epsilon n^{1/4}\right)+ P\left(\eta_{n^{1-\beta}k_1}^n(x_{k_1,k_2}^2)\geq \tfrac{1}{8}\epsilon n^{1/4}\right)\bigr\} \biggr\}.
\label{bound}
\end{split}
\end{equation}

The probabilities $ P(N_1\geq 1)$ and $ P(N_2\geq cn^{1/2+\alpha})$ vanish as $n\to\infty$ by lemma \ref{lemman1} and \eqref{N_2}. $\Pi(cn^{3/2+\alpha-\beta})$ is stochastically bounded by a sum of $M_n=[cn^{3/2+\alpha-\beta}]+1$ i.i.d. mean 1 Poisson variables, and so a standard large deviation estimate gives 
$$ P\left(\Pi(cn^{3/2+\alpha-\beta})\geq \tfrac{1}{4}\epsilon n^{1/4}\right)\leq \exp\left\{-M_nI(\tfrac{1}{4}M_n^{-1}\epsilon n^{1/4})\right\},$$
where $I$ is the Cram\'{e}r rate function for the Poisson(1) distribution. By the choice of $\alpha$ and $\beta$, $M_n\geq n^\alpha$, while $M_n^{-1}n^{1/4}\to\infty$. Consequently, there are constants $0<C_0,C_1<\infty$ such that
 $$\sum_{k_1=0}^{[Tn^\beta]}\sum_{k_2=[-Sn^\beta]}^{[Sn^\beta]} P\left(\Pi(cn^{3/2+\alpha-\beta})\geq \tfrac{1}{4}\epsilon n^{1/4}\right)\leq C_0n^{2\beta} \exp\{-C_1n^\alpha\}\to 0.$$
By Lemma 4.10 in \cite{SEPP}, we have 
\begin{equation}
\label{occupation}
\sup_{x\in\mathbb Z,t\geq 0}E[\eta_t(x)^{12}]<\infty.
\end{equation}
So,
\begin{align*}&\sum_{k_1=0}^{[Tn^\beta]}\sum_{k_2=[-Sn^\beta]}^{[Sn^\beta]}P(\eta_{k_1n^{1-\beta}}^n(x_{k_1,k_2}^i)\geq \tfrac{1}{8}\epsilon n^{1/4})\\
&\qquad \qquad\qquad\qquad\leq (Tn^\beta+1)(2Sn^\beta+1)8^{12}\epsilon^{-12}n^{-3}\sup_{x,t,n}E[\eta_t^n(x)^{12}]\end{align*}
vanishes as $n\to\infty$ by (\ref{occupation}) and because $2\beta-3<0$ .\\
Similarly for the other probability in \eqref{bound}.

\end{proof}
 
Since the two conditions of Theorem \ref{criterion} hold for the sequence of processes $\{n^{-1/4}\bar{Y}_n\}$, we can conclude that
\begin{equation}
\label{tight} 
\lim_{\delta\downarrow 0}\limsup_n P\{w_{n^{-1/4}\bar{Y}_n}(\delta)\geq \epsilon\}=0 \text{ for all $\epsilon>0$.}
\end{equation}
\subsection{\bf Weak Convergence}
 Finally, we use the theorem about weak convergence in $D_2$ from \cite{BW}. By Theorem 2 in \cite{BW} we have $X_n$ converges weakly to $X$ in $D_2$ if,
\begin{enumerate}
\item
$(X_n(t_1,r_1),\cdots,X_n(t_N,r_N))$ converges weakly to $(X(t_1,r_1),\cdots,X(t_N,r_N))$ for all finite subsets $\{(t_i,r_i)\}\in [0,T]\times [-S,S]$, and
\item
$\lim_{\delta\to 0}\limsup_n P\{w_{X_n}(\delta)\geq \epsilon\}=0$ for all $\epsilon>0$,
where
\[w_x(\delta)=\sup_{\substack{(s,q),(t,r)\in[0,T]\times [-S,S]\\|(s,q)-(t,r)|<\delta}}|x(s,q)-x(t,r)|.\]
\end{enumerate}
This, together with the convergence of finite-dimensional distributions of $\{n^{-1/4}\bar{Y}_n(\cdot,\cdot)\}$ and (\ref{tight}) gives us weak convergence of $\{n^{-1/4}\bar{Y}_n(\cdot,\cdot)\}$ as $n\to\infty$. Since the expectations $n^{-1/4}EY_n(t,s)$ vanish uniformly over $0\leq t\leq T$, $0\leq |s|\leq S$ by (\ref{meanvanish}), we conclude that the process $\{n^{-1/4}Y_n(\cdot,\cdot)\}$ converges weakly as $n\to\infty$.

\section{\bf Proof of large deviation results}\label{sec:largedev}
\begin{proof}[\bf Proof of Theorem \ref{ld} and 
Corollary \ref{ldpcor}]
Assume that $\eta_0^n(m),m\in\mathbb Z$ are i.i.d. Fix $r\in\mathbb R$ and $t>0$.We prove that $n^{-1/2}Y_n(t,r)$ satisfies the LDP with a good rate function. 
We start with some preliminary calculations.
\begin{equation}
\begin{split}
Y_n(t,r)&=\sum_{m=-\infty}^\infty\sum_{j=1}^{\eta_0^n(m)}\bigl[\mathbf{1} \{X_{m,j}(nt)\leq [nvt]+[r\sqrt{n}]\}\mathbf{1} \{m>[r\sqrt{n}]\}\\
&-\mathbf{1} \{X_{m,j}(nt)>[nvt]+[r\sqrt{n}]\}\mathbf{1} \{m\leq [r\sqrt{n}]\}\bigr]\\
&=\sum_{m=-\infty}^\infty \sum_{j=1}^{\eta_0^n(m)} [f_{m,j}^{n,(1)}(t,r)-f_{m,j}^{n,(2)}(t,r)].
\end{split}
\end{equation}

Define
\begin{equation}
\begin{split}
M_m^n(\lambda)&
=Ee^{\lambda [f_{m,1}^{n,(1)}(t,r)-f_{m,1}^{n,(2)}(t,r)]}.
\end{split}
\end{equation}
If $m>[r\sqrt{n}]$ then
\begin{equation}\label{mgf1-1}
\begin{split}
M_m^n(\lambda)&=Ee^{\lambda f_{m,1}^{n,(1)}(t,r)}=E\left[\sum_{k\geq 0} \frac{(\lambda f_{m,1}^{n,(1)}(t,r))^k}{k!}\right]\\
&=1+(e^\lambda-1)E[f_{m,1}^{n,(1)}(t,r)]\\&=1+(e^\lambda-1)P(X(nt)\leq [nvt]+[r\sqrt{n}]-m)
\end{split}
\end{equation}
where $X(\cdot)$ represents a random walk with rates $p(x)$ starting at the origin.
Similarly, if $m\leq [r\sqrt{n}]$ then
\begin{equation}\label{mgf1-2}
\begin{split}
M_m^n(\lambda)&=Ee^{-\lambda f_{m,1}^{n,(2)}(t,r)}=E\left[\sum_{k\geq 0} \frac{(-\lambda\cdot f_{m,1}^{n,(2)}(t,r))^k}{k!}\right]\\&=1+(e^{-\lambda}-1)P(X(nt)> [nvt]+[r\sqrt{n}]-m).
\end{split}
\end{equation}
We now calculate the logarithmic moment generating function for $Y_n(t,r)$.

\begin{equation*}
\begin{split}
\log Ee^{\lambda Y_n(t,r)}&=\sum_{|m-[r\sqrt{n}]|\leq [nt\delta]}\log E\exp\left\{\lambda\sum_{j=1}^{\eta_0^n(m)}[f_{m,j}^{n,(1)}(t,r)-f_{m,j}^{n,(2)}(t,r)]\right\}\\
&+\sum_{|m-[r\sqrt{n}]|> [nt\delta]}\log E\exp\left\{\lambda\sum_{j=1}^{\eta_0^n(m)}[f_{m,j}^{n,(1)}(t,r)-f_{m,j}^{n,(2)}(t,r)]\right\}\intertext{(By large deviation bounds on $X(nt)$   \eqref{ld_bounds_rw1} and \eqref{ld_bounds_rw2}
, the second term is of $o(\sqrt{n})$)}\\
&=\sum_{|m-[r\sqrt{n}]|\leq [nt\delta]}\log E\exp\left\{\lambda\sum_{j=1}^{\eta_0^n(m)}[f_{m,j}^{n,(1)}(t,r)-f_{m,j}^{n,(2)}(t,r)]\right\}+o(\sqrt{n})
\end{split}
\end{equation*}

\begin{equation}\label{mgf2}
\begin{split}
\lim_{n\to\infty}\frac{1}{\sqrt{n}}\log Ee^{\lambda Y_n(t,r)}&=\lim_{n\to\infty}\frac{1}{\sqrt{n}}\sum_{|m-r\sqrt{n}]|\leq [nt\delta]}\log E\left[e^{\lambda\sum_{j=1}^{\eta_0^n(m)}[f_{m,j}^{n,(1)}(t,r)-f_{m,j}^{n,(2)}(t,r)]}\right]\\&=\lim_{n\to\infty}\frac{1}{\sqrt{n}}\sum_{|m-r\sqrt{n}]|\leq [nt\delta]}\log \sum_{k\geq 0}P(\eta_0^n(m)=k)(M_m^n(\lambda))^k
\intertext{(Recall that $\gamma(\alpha)=\log Ee^{\alpha\eta_0^n(\cdot)}$)}
&=\lim_{n\to\infty}\frac{1}{\sqrt{n}}\sum_{|m-r\sqrt{n}]|\leq [nt\delta]} \gamma(\log M_m^n(\lambda)).
\end{split}
\end{equation}
Using (\ref{mgf1-1}), (\ref{mgf1-2}), (\ref{mgf2}), applying Central Limit Theorem and a Riemann sum argument 
we get,
\begin{equation}
\begin{split}
\Lambda(\lambda)=&\lim_{n\to\infty}\frac{1}{\sqrt{n}}\log Ee^{\lambda Y_n(t,r)}\\&=\int_0^\infty\gamma(\log\{1+(e^\lambda-1)\Phi_{\kappa_2t}(-x)\})dx+\int_{-\infty}^0\gamma(\log\{1+(e^{-\lambda}-1)\Phi_{\kappa_2t}(x)\})dx.
\end{split}
\end{equation}

By Assumption \ref{assmpnld} we get $\Lambda(\lambda)<\infty$ for $\lambda\in\mathbb R$. It is also easy to check that $\Lambda(\lambda)$ is strictly convex and essentially smooth on $\mathbb R$.
By Theorem 2.3.6 in \cite{DZ} (G\"artner-Ellis theorem) $I(\cdot)$, the convex dual of $\Lambda(\lambda)$, is the good rate function. We now find the explicit expression for the rate function.

 If $\eta_0^n(m)\sim$ Poisson$(\rho)$, then $\gamma(\alpha)=\rho(e^\alpha-1)$.
Therefore,
\begin{align*}\Lambda(\lambda)&=\rho(e^\lambda-1)\int_0^\infty \Phi_{\kappa_2t}(-x)dx+\rho(e^{-\lambda}-1)\int_{-\infty}^0 \Phi_{\kappa_2t}(x)dx\\&=\rho\sqrt{\frac{\kappa_2t}{2\pi}}(e^\lambda+e^{-\lambda}-2).
\end{align*}
The convex dual of this is:
\begin{align*}I(x)&=\sup_{\lambda\in\mathbb R}\{x\lambda-\Lambda(\lambda)\}\\&
=x\log\biggl(\frac{x\sqrt{\pi}}{\rho\sqrt{2\kappa_2t}}+\sqrt{1+\frac{x^2\pi}{2\rho^2\kappa_2t}}\biggr) -\rho\sqrt{\frac{2\kappa_2t}{\pi}}\biggl(\sqrt{1+\frac{x^2\pi}{2\rho^2\kappa_2t}}-1\biggr)\text{ for }x\in\mathbb R.
\end{align*}
This proves \eqref{rate_fn}.

To prove Theorem \ref{ld} we first check that $I(\cdot)$ given by \eqref{RateFn} is the convex dual of $\Lambda(\cdot)$ and then outline how to get the expression in \eqref{RateFn}.
Elementary but tedious computations give
\begin{equation}\label{I'}I'(\Lambda'(\lambda))=\lambda.\end{equation}
It can be shown that $\Lambda'$ is continuous and strictly increasing. Therefore, $I'$ is defined on the whole real line  by \eqref{I'}. 
By Theorem 26.5 in \cite{Rock} (page 258), $I$ must be $\Lambda^*$ plus a constant. But $I(0)=0=\Lambda^*(0)$, so $I=\Lambda^*$. This proves that $I$ given by \eqref{RateFn} is the convex dual of $\Lambda$ and hence the rate function.  Since $\Lambda'$ is continuous and strictly increasing, $I'$ must be strictly increasing from \eqref{I'} and hence $I$ must be strictly convex. This completes the proof 
of Theorem \ref{ld}.  Corollary \ref{ldpcor} comes as a special
case of \eqref{RateFn}. 

For the reader's benefit, here is an indication of how 
the expression in \eqref{RateFn} is derived 
non-rigorously.
We first approximate the integral in $\Lambda(\cdot)$ by a Riemann sum. 
\[\Lambda(\lambda)=\lim_{\delta\to 0}\sum_k \tilde{\Lambda}_k^\delta(\lambda)\]
where 
\begin{equation}
\label{summands}
 \tilde{\Lambda}_k^\delta(\lambda) := \begin{cases}\delta\gamma\bigl( Br_{1-\P(k\delta)}(\lambda)\bigr)&\text{ for }k>0\\
 \delta\gamma\bigl( Br_{\P(k\delta)}(-\lambda)\bigr)&\text{ for }k\leq 0.
 \end{cases}
 \end{equation}
The function $Br_p(\lambda)$ denotes the log moment generating function for Bernoulli random variables as defined in section 2.3. 
Observe that the summands \eqref{summands} are a composition of two functions. 
Using the definition of convex dual, it is easy to prove the identity
\begin{equation}
\label{composition}
(f\circ g)^*(x)=f'(g(\lambda))g^*(g'(\lambda))+f^*(f'(g(\lambda)))\end{equation}
where $\lambda$ is such that \begin{equation}\label{constraint}x=f'(g(\lambda))\cdot g'(\lambda).\end{equation}
We use this identity to get the convex dual of the summands \eqref{summands}. 

For small $\delta$
\[\Lambda(\lambda)\approx \sum_{|k|\leq [1/\delta]} \tilde{\Lambda}_k^\delta(\lambda).\]
 The convex dual of the sum is then given as an infimal convolution. 
\begin{equation}\label{inf_conv}
\begin{split}\Lambda^*(x)&\approx \left(\sum_{|k|\leq [1/\delta]}\tilde{\Lambda}_k^\delta\right)^*(x)=\inf_{\sum_{|k|\leq [1/\delta]}x_k=x}\sum_{|k|\leq [1/\delta]}\left(\tilde{\Lambda}_k^\delta\right)^*(x_k)\intertext{now using \eqref{composition} we get}
&=\inf_{\sum\limits_{|k|\leq [1/\delta]}x_k=x}\Biggl\{\delta\sum_{k=1}^{[1/\delta]}\Bigl[\gamma'\left(Br_{1-\P(k\delta)}(\lambda_k)\right)Br_{1-\P(k\delta)}^*(Br_{1-\P(k\delta)}'(\lambda_k))\\
&\qquad\qquad+\gamma^*\left(\gamma'(Br_{1-\P(k\delta)}(\lambda_k))\right)\Bigr]\\
&\quad +\delta \sum_{k=-[1/\delta]}^0\Bigl[\gamma'\left(Br_{\P(k\delta)}(-\lambda_k)\right)Br_{\P(k\delta)}^*(Br_{\P(k\delta)}'(-\lambda_k))\\
&\qquad\qquad+\gamma^*\left(\gamma'(Br_{\P(k\delta)}(-\lambda_k))\right)\Bigr]\Biggr\}
\end{split}
\end{equation}
where \[x_k= \begin{cases}\delta\gamma'\left(Br_{1-\P(k\delta)}(\lambda_k)\right)Br'_{1-\P(k\delta)}(\lambda_k)&\text{ for }k>0\\
-\delta\gamma'\left(Br_{\P(k\delta)}(-\lambda_k)\right)Br'_{\P(k\delta)}(-\lambda_k)&\text{ for }k\leq 0.\end{cases}\]

Note that
\begin{gather}
Br_{1-\P(y)}(\lambda)=Z_\lambda(y) \text{ for }y>0\ ,\ 
Br_{\P(y)}(-\lambda)=Z_\lambda(y) \text{ for }y\leq 0\\
Br'_{1-\P(y)}(\lambda)=1-F_\lambda(y)\text{ for } y>0\ ,\ 
Br'_{\P(y)}(-\lambda)=F_\lambda(y) \text{ for }y\leq 0.
\end{gather}

Taking the limit as $\delta\to 0$  of \eqref{inf_conv}  we get the constrained variational problem
%. The constraint $\sum_{|k|\leq [1/\delta]}x_k=x$ will correspond to \eqref{current_x} . 
\begin{equation*}
\begin{split}
\Lambda^*(x)&=\inf_{{\displaystyle\{\lambda: \int\limits_{-\infty}^\infty \tfrac{\partial}{\partial \lambda(y)}\left[\gamma\left(Z_{\lambda(y)}(y)\right)\right] dy=x\}}}\Biggl[ \int_{-\infty}^\infty\gamma^*\left\{\gamma'(Z_{\lambda(y)}(y))\right\}dy\\
&\qquad+\int_{-\infty}^\infty \gamma'(Z_{\lambda(y)}(y))Br^*_{\P(y)}(F_{\lambda(y)}(y))dy\Biggr].
\end{split}
\end{equation*}
Solving the above variational problem using standard functional analysis techniques we get $\lambda(y)\equiv \alpha(x)$ minimizes the above functional and $\alpha(x)$ satisfies the constraint \eqref{current_x}. 
\end{proof}

\begin{proof}[\bf Proof of Theorem \ref{LDP_process}]
Let $Y_n(t)=Y_n(t,0)$. Under the assumption $\eta_0^n(m)\sim Poisson(\rho)$ we can show that $\{n^{-1/2}Y_n(\cdot)\}$ satisfies the large deviation principle in $D_{\mathbb R}[0,\infty)$. 

Fix $k$ time points $0\leq t_1<t_2<\cdots <t_k$.
Define the $k$-vectors with $0,1$ entries by \[\vec{F}_{m,j}=\bigl({\bf 1}\{X_{m,j}(nt_1)\leq [nvt_1]\},\hdots ,{\bf 1}\{X_{m,j}(nt_k)\leq [nvt_k]\}\bigr)\] and
\[\vec{G}_{m,j}=\bigl({\bf 1}\{X_{m,j}(nt_1)> [nvt_1]\},\hdots ,{\bf 1}\{X_{m,j}(nt_k)> [nvt_k]\}\bigr).\]
Then \[\bigl(Y_n(t_1),\hdots,Y_n(t_k)\bigr)=\sum_{m=1}^\infty\sum_{j=1}^{\eta_0^n(m)}\vec{F}_{m,j}-\sum_{m=-\infty}^0\sum_{j=1}^{\eta_0^n(m)}\vec{G}_{m,j}.\]
Let $\vec{u}:=(u^{(1)},\hdots,u^{(k)})\in\{0,1\}^k$ be a $k$-vector with $\vec{u}\neq \vec{0}$.
Define \[N^1_n(\vec{u}):=\sum_{m=1}^\infty\sum_{j=1}^{\eta_0^n(m)}{\bf 1}\{\vec{F}_{m,j}=\vec{u}\}\] and
\[N_n^2(\vec{u}):=\sum_{m=-\infty}^0\sum_{j=1}^{\eta_0^n(m)}{\bf 1}\{\vec{G}_{m,j}=\vec{u}\}.\]
If $\eta_0^n(\cdot)$ are i.i.d. Poisson($\rho$) random variables, then $N_n^1(\vec{u})$ is a Poisson random variable with rate
  \[\sum_{m=1}^\infty\rho P(\bigcap_{i=1}^kC^{u^{(i)}}_{m,i})<\infty,\] where 
$C^1_{m,i}=\{X(nt_i)\leq [nvt_i]-m\}$,
$C^0_{m,i}=\{X(nt_i)> [nvt_i]-m\}$,
and  $N_n^2(\vec{u})$ is a Possion random variable with rate 
 \[\sum_{m=-\infty}^0\rho P(\bigcap_{i=1}^kD^{u^{(i)}}_{m,i})<\infty,\] where 
$D^1_{m,i}=\{X(nt_i)> [nvt_i]-m\}$,
$D^0_{m,i}=\{X(nt_i)\leq [nvt_i]-m\}$. We use the bounds \eqref{ld_bounds_rw1} and \eqref{ld_bounds_rw2} on the large deviations of random walks to justify the rates being finite.

We can write \[(Y_n(t_1),\hdots, Y_n(t_k))=\sum_{\vec{u}\neq\vec{0}}(N_n^1(\vec{u})-N_n^2(\vec{u}))\vec{u}.\] Let $\vec{u}_j=(u_j^{(1)},\hdots ,u_j^{(k)}), j=1,\hdots , 2^k-1$ denote the $\{0,1\}$-valued $k$-vectors, excluding the zero vector. By the contraction principle (Theorem 4.2.1 in \cite{DZ}) we can conclude that $(Y_n(t_1),\hdots, Y_n(t_k))$ satisfies the LDP with good rate function given by \[I_{t_1,\hdots,t_k}(\vec{x}):=\inf J(y_1,\hdots,y_{2^k-1},z_1,\hdots,z_{2^k-1})\]
for any $\vec{x}=(x_1,\hdots,x_k)$. The $\inf$ here is taken over the set $\{(y_1,\hdots,y_{2^k-1},z_1,\hdots,z_{2^k-1}):\vec{x}=\sum_{\vec{u}_j\neq\vec{0}}(y_j-z_j)\vec{u}_j\}$. $J(y_1,\hdots,y_{2^k-1},z_1,\hdots,z_{2^k-1})$ is the good rate function for the sequence of vectors $\{(N_n^1(\vec{u_1}),\hdots,N_n^1(\vec{u}_{2^k-1}),N_n^2(\vec{u_1}),\hdots,N_n^2(\vec{u}_{2^k-1}))\}_n$ of $2^{k+1}-2$ independent Poisson random variables.
Define for $x\geq0$, \[C_{x,j}^1=\{B(\kappa_2t_j)\leq -x\}, \ 
 C_{x,j}^0=\{B(\kappa_2t_j)> -x\},\] 
\[D_{x,j}^1=\{B(\kappa_2t_j)> x\}\text{ and }
D_{x,j}^0=\{B(\kappa_2t_j)\leq x\},\] where $B(\cdot)$ is standard Brownian motion.
\[J(y_1,\hdots,y_{2^k-1},z_1,\hdots,z_{2^k-1})=\sum_{i=1}^{2^k-1}\bigl[y_i\log \frac{y_i}{\alpha_i}+z_i\log\frac{z_i}{\beta_i}-\alpha_i(\frac{y_i}{\alpha_i}-1)-\beta_i(\frac{z_i}{\beta_i}-1)\bigr]\]
where \begin{align*}\alpha_i=\lim_{n\to\infty}\frac{1}{\sqrt{n}}EN_n^1(\vec{u_i})=\rho\int_0^\infty P\left(\bigcap_{j=1}^kC_{x,j}^{{u_i}^{(j)}}\right)dx\end{align*}
and 
 \begin{align*}\beta_i=\lim_{n\to\infty}\frac{1}{\sqrt{n}}EN_n^2(\vec{u_i})=\rho\int_0^\infty P\left(\bigcap_{j=1}^kD_{x,j}^{{u_i}^{(j)}}\right)dx.\end{align*}
 
 We now apply Theorem 4.30 in \cite{KF} to $\{n^{-1/2}Y_n(\cdot)\}$. This gives us the large deviation principle for $\{n^{-1/2}Y_n(\cdot)\}$ in $D_{\mathbb R}[0,\infty)$ with good rate function \begin{equation}\label{rate_process}I(x)=\sup_{\{t_i\}} I_{t_1,\hdots,t_m}(x(t_1),\hdots,x(t_m)).\end{equation}
\end{proof}

{\bf Acknowledgements}
 This paper is part of my Ph.D. thesis. I would like to thank my advisor Prof. Timo Sepp\"{a}l\"{a}inen for suggesting the problem and for his help in guiding me in my research.\\

\bibliographystyle{plain}
\bibliography{bib-2}

\begin{thebibliography}{10}

\bibitem{BRS}
M{\'a}rton Bal{\'a}zs, Firas Rassoul-Agha, and Timo Sepp{\"a}l{\"a}inen.
\newblock The random average process and random walk in a space-time random
  environment in one dimension.
\newblock {\em Comm. Math. Phys.}, 266(2):499--545, 2006.

\bibitem{BW}
P.~J. Bickel and M.~J. Wichura.
\newblock Convergence criteria for multiparameter stochastic processes and some
  applications.
\newblock {\em Ann. Math. Statist.}, 42:1656--1670, 1971.

\bibitem{CD}
J.~Theodore Cox and Richard Durrett.
\newblock Large deviations for independent random walks.
\newblock {\em Probab. Theory Related Fields}, 84(1):67--82, 1990.

\bibitem{CG}
J.~Theodore Cox and David Griffeath.
\newblock Large deviations for {P}oisson systems of independent random walks.
\newblock {\em Z. Wahrsch. Verw. Gebiete}, 66(4):543--558, 1984.

\bibitem{DZ}
Amir Dembo and Ofer Zeitouni.
\newblock {\em Large deviations techniques and applications}.
\newblock Jones and Bartlett Publishers, Boston, MA, 1993.

\bibitem{Durr}
Detlef D{\"u}rr, Sheldon Goldstein, and Joel~L. Lebowitz.
\newblock Asymptotics of particle trajectories in infinite one-dimensional
  systems with collisions.
\newblock {\em Comm. Pure Appl. Math.}, 38(5):573--597, 1985.

\bibitem{KF}
Jin Feng and Thomas~G. Kurtz.
\newblock {\em Large deviations for stochastic processes}, volume 131 of {\em
  Mathematical Surveys and Monographs}.
\newblock American Mathematical Society, Providence, RI, 2006.

\bibitem{KL}
Claude Kipnis and Claudio Landim.
\newblock {\em Scaling limits of interacting particle systems}, volume 320 of
  {\em Grundlehren der Mathematischen Wissenschaften [Fundamental Principles of
  Mathematical Sciences]}.
\newblock Springer-Verlag, Berlin, 1999.

\bibitem{Lee}
Tzong-Yow Lee.
\newblock Large deviations for independent random walks on the line.
\newblock {\em Ann. Probab.}, 23(3):1315--1331, 1995.

\bibitem{MS}
Peter March and Timo Sepp{\"a}l{\"a}inen.
\newblock Large deviations from the almost everywhere central limit theorem.
\newblock {\em J. Theoret. Probab.}, 10(4):935--965, 1997.

\bibitem{Sethuraman}
Magda Peligrad and Sunder Sethuraman.
\newblock On fractional brownian motion limits in one dimensional
  nearest-neighbor symmetric simple exclusion.
\newblock \url{http://www.citebase.org/abstract?id=oai:arXiv.org:0711.0017},
  2007.

\bibitem{Rock}
R.~Tyrrell Rockafellar.
\newblock {\em Convex analysis}.
\newblock Princeton Mathematical Series, No. 28. Princeton University Press,
  Princeton, N.J., 1970.

\bibitem{SEPP}
Timo Sepp{\"a}l{\"a}inen.
\newblock Second-order fluctuations and current across characteristic for a
  one-dimensional growth model of independent random walks.
\newblock {\em Ann. Probab.}, 33(2):759--797, 2005.

\bibitem{Spohn}
H.~Spohn.
\newblock {\em Large scale dynamics of interacting particles}.
\newblock Texts and monographs in physics. Springer-Verlag, Berlin, 1991.

\end{thebibliography}
\end{document}